\RequirePackage[polutonikogreek,english]{babel}
\documentclass[a4paper, 11pt]{amsart}

\usepackage[dvipdfm, backref=page,colorlinks,bookmarksopen=true]{hyperref}
\usepackage{amsfonts,amsmath,amssymb}
\usepackage{amsthm}
\usepackage{mathrsfs}
\usepackage{xspace}
\usepackage{enumerate}
\usepackage{graphicx,psfrag}
\usepackage{a4wide}
\usepackage[latin1]{inputenc}  
\usepackage[T1]{fontenc}       
\theoremstyle{plain}
\newtheorem{proposition}{Proposition}[section]
\newtheorem{theorem}[proposition]{Theorem}

\newtheorem{thmintro}{Theorem}

\newtheorem{lemma}[proposition]{Lemma}
\newtheorem{corollary}[proposition]{Corollary}

\theoremstyle{definition}
\newtheorem{definition}[proposition]{Definition}

\theoremstyle{remark}
\newtheorem{remark}[proposition]{Remark}
\newtheorem{example}[proposition]{Example}

\everymath{\displaystyle}
\newcommand{\RLF}{\mathrm{Rad}_{\mathrm{LF}}}
\DeclareMathOperator{\ch}{Ch} %
\DeclareMathOperator{\cl}{Res_1}%
\DeclareMathOperator{\Res}{Res_{\mathrm{sph}}}%
\DeclareMathOperator{\Fix}{Fix}%
\DeclareMathOperator{\Stab}{Stab}%
\DeclareMathOperator{\Conv}{Conv} %
\DeclareMathOperator{\proj}{proj} %
\DeclareMathOperator{\projc}{\pi_{\mathrm{Ch}}} %
\DeclareMathOperator{\projr}{\pi_{\mathrm{Res}}} %
\DeclareMathOperator{\Cc}{\mathscr{C}_1} %
\DeclareMathOperator{\Cr}{\mathscr{C}_{\mathrm{sph}}} %
\DeclareMathOperator{\Ch}{\mathscr{C}_{\mathrm{horo}}} %
\DeclareMathOperator{\Cg}{\mathscr{C}_{\mathrm{gp}}} %
\DeclareMathOperator{\Cu}{\mathscr{C}_{\mathrm{ultra}}} %
\DeclareMathOperator{\SL}{SL}
\DeclareMathOperator{\GL}{GL}
\DeclareMathOperator{\St}{St} %
\DeclareMathOperator{\Aut}{Aut}%
\DeclareMathOperator{\Isom}{Isom}%

\newcommand{\cat}{{\upshape CAT(0)}\xspace}
\newcommand{\bd}{\partial_\infty}
\newcommand{\bdfine}{\partial_\infty^\mathrm{fine}}
\newcommand{\Axi}{\mathcal A_\xi}
\newcommand{\HAxi}{ {}^{\frac 1  2} \! \mathcal A_\xi}

\newcommand{\R}{\mathbb{R}}                          
\newcommand{\N}{\mathbb{N}}                          
\newcommand{\Z}{\mathbb{Z}}                          

\begin{document}
\title{Combinatorial and group-theoretic compactifications of buildings}

\author{Pierre-Emmanuel Caprace$^*$}
\address{Universit\'e catholique de Louvain, D\'epartement de Math\'ematiques, Chemin du Cyclotron 2, 1348 Louvain-la-Neuve, Belgium}
\email{pe.caprace@uclouvain.be}
\thanks{*F.N.R.S. Research Associate}

\author{Jean L\'ecureux}
\address{Université de Lyon;
Université Lyon 1;
INSA de Lyon;
Ecole Centrale de Lyon;
CNRS, UMR5208, Institut Camille Jordan,
43 blvd du 11 novembre 1918,
F-69622 Villeurbanne-Cedex, France}
\email{lecureux@math.univ-lyon1.fr}


\begin{abstract}
Let $X$ be a building of arbitrary type. A compactification $\Cr(X)$
of the set $\Res(X)$ of spherical residues of $X$ is introduced. We
prove that it coincides with the horofunction compactification of
$\Res(X)$ endowed with a natural combinatorial distance which we
call the \emph{root-distance}. Points of $\Cr(X)$ admit amenable
stabilisers in $\Aut(X)$ and conversely, any amenable subgroup
virtually fixes a point in $\Cr(X)$. In addition, it is shown that,
provided $\Aut(X)$ is transitive enough, this compactification also
coincides with the group-theoretic compactification constructed
using the Chabauty topology on closed subgroups of $\Aut(X)$. This
generalises to arbitrary buildings results established by
Y.~Guivarc'h and B.~R\'emy \cite{GuR} in the Bruhat--Tits case.

\end{abstract}

\maketitle
\tableofcontents

\section*{Introduction}

The best known and probably most intuitively obvious
compactification of a non-compact Riemannian symmetric space $M$ is
the \textbf{visual compactification} $\overline M = M \cup \bd M$,
whose points at infinity consist in equivalence classes of geodesic
rays at finite Hausdorff distance of one another. Following
Gromov~\cite{BGS}, this compactification may be identified with the
\textbf{horofunction compactification} $\Ch(M)$, whose points at
infinity are Busemann functions. This canonical identification holds
in fact for any \cat metric space, see \cite[Theorem~II.8.13]{BH}.

Another way to approach to visual compactification of $M$ is the
following. Using the visual map which associates to every pair of
points $p, q \in M$ the direction at~$p$ of the geodesic segment
$[p, q]$, it is possible to associate to every point of $M$ a unique
element of the unit tangent ball bundle over $M$. The total space of
this bundle being compact, one obtains a compactification by passing
to the closure of the image of $M$; this coincides with the visual
compactification $\overline M$. Here again, the construction has a
natural analogue which makes sense in any locally compact \cat space
$X$ provided that the space of directions $\Sigma_p X$ at every
point $p \in X$ is compact. This condition is automatically
satisfied if $X$ is geodesically complete (\emph{i.e.} every
geodesic segment may be extended to a bi-infinite geodesic line,
which need not be unique) or if $X$ has the structure of a \cat cell
complex. In the latter case, each space of direction $\Sigma_p X$ is
endowed with the structure of a finite cell complex.

This suggests to modify the above construction of the visual
compactification as follows. Assume $X$ is a locally finite \cat
cell complex. Then the space of direction $\Sigma_p X$ has a
cellular structure; one denotes by $\St(p)$ the corresponding set of
cells. Associating to each point its support, one obtains a
canonical map  $\Sigma_p X \to \St(p)$. Pre-composing with the
afore-mentioned visual map, one obtains a map $X \to \prod_{p \in X}
\St(p)$. The closure of this map is called the \textbf{combinatorial
compactification} of $X$. It should be noted that the above map is
not injective in general: two points with the same support are
identified.

\medskip
The main purpose of this paper is to pursue this line of thoughts in
the special case of \emph{buildings} of arbitrary type. Similar
developments in the case of \cat cube complexes are carried out in
the Appendix.

In the case of buildings, the relevant simplices are the so-called
\textbf{residues of spherical type}, also called \textbf{spherical
residues} for short. The above combinatorial compactification thus
yields a compactification of the set $\Res(X)$ of all spherical
residues, and the above `visual map' $\Res(X) \to \prod_{\sigma \in
\Res(X)} \St(\sigma)$ may be canonically defined in terms of the
\textbf{combinatorial projection}. Its closure is the the combinatorial compactification and will be denoted by $\Cr(X)$. 

The set $\Res(X)$ may moreover be endowed in the canonical with the
structure of a discrete metric space. For example, a graph structure
on $\Res(X)$ is obtained by declaring two residues adjacent if one
is contained in the other. We shall introduce a sligthly different
distance, called the \textbf{root-distance} which has the advantage
that its restriction to the chamber-set $\ch(X)$ coincides with the
gallery distance (see Section~\ref{sec:root-distance}). As any
proper metric space, the discrete metric space $\Res(X)$ admits a
\textbf{horofunction compactification} This turns out to coincide with the
combinatorial compactification (see Theorem~\ref{thm:horo}).

It is important to remark that the combinatorial compactification
does not coincide with the visual one. Although there are elementary
ways to establish the latter fact, strong evidence is provided by
the following result (see Theorem~\ref{moyennfix}).

\begin{thmintro}\label{amenable:intro}
Let $X$ be a locally finite building. Then every amenable subgroup
of $\Aut(X)$ has a finite index subgroup which fixes  some point in
$\Cr(X)$.

Conversely, the full stabiliser of every point of $\Cr(X)$ is a
closed amenable subgroup.
\end{thmintro}

In the special case of Bruhat--Tits buildings, a similar statement
was established in~\cite[Theorem~33]{GuR} using another compactification called the group-theoretic compactification. 
The construction of the latter goes back to an idea of
Y.~Guivarc'h in the case of symmetric spaces and be be outlined as follows. A symmetric
space $M$ embeds in the space of closed subgroups of $\Isom(M)$ by
attaching to each point its isotropy group. Since the space of
closed subgroups endowed with Chabauty topology is compact, one
obtains a compactification by passing to the closure. This yields
the \textbf{group-theoretic compactification } $\Cg(M)$. This turns
out to be equivariantly isomorphic to the maximal Satake and
Furstenberg compactifications (see  \cite{GJT}, \cite{BJ}). In the
case of buildings, since points with the same support have identical
stabilisers, this approach cannot offer better than a
compactification of the set $\Res(X)$.

\begin{thmintro}\label{chab:intro}
Assume that $\Aut(X)$ acts strongly transitively. The
group-theoretic compactification  $\Cg(X)$ is
$\Aut(X)$-equivariantly homeomorphic to the maximal combinatorial
compactification $\Cr(X)$. More precisely, a sequence $(R_n)$ of
spherical residues converges to some $\xi \in \Cr(X)$ if and only if
the sequence of their stabilisers $(G_{R_n})$ converges to the
locally finite radical of $G_\xi$ in the Chabauty topology.
\end{thmintro}

Recall that the \textbf{locally finite radical} of a locally compact
group $G$ is the unique subgroup $\RLF(G)$ which is
\textbf{(topologically) locally finite} (\emph{i.e.} all of whose
finitely generated subgroups are relatively compact), normal and
maximal for these properties. It was shown in~\cite{Cap} that a
closed subgroup $H$ of $\Aut(X)$ is amenable if and only if $H/
\RLF(H)$ is virtually Abelian.

It is shown in~\cite{GuR} that the group-theoretic
compactifications may be canonically identified with the
\textbf{polyhedral compactification} constructed by E.~Landvogt
in~\cite{Lan}. Theorem~\ref{chab:intro} may be viewed as an
extension of this to the case of arbitrary buildings.

\medskip
A central tool introduced in this work to study the combinatorial
compactification is the notion of \textbf{combinatorial sectors},
which extend to the general case the classical notion of sectors in
Bruhat--Tits theory. Given a point $\xi \in \Cr(X)$, we associate to
every $x \in \Res(X)$ as sector $Q(x, \xi)$ based at $x$ and
pointing to $\xi$ (see Section~\ref{sec:sectors}). Every sector is
contained in an apartment; the key property is that the collection
of all sectors pointing to $\xi \in \Cr(X)$ is \textbf{filtering};
in other words any two sectors pointing to $\xi$ contain a common
subsector (see Proposition~\ref{mmapp2}).

We emphasize that all of our considerations are valid for arbitrary
buildings and are of elementary nature; in particular, no use is
made of the theory of algebraic groups. Moreover, as it will appear
in the core of the paper, most of the results remain valid for
buildings which are not necessarily locally finite (in that case,
one uses the term \emph{bordification} instead of
\emph{compactification}).

\medskip
The paper is organised as follows. In a first section we introduce
and study the properties of a combinatorial distance on $\Res(X)$
which we call the root-distance. The next section is devoted to the
combinatorial compactification. Combinatorial sectors are introduced
and used to prove that every point of $\Cr(X)$ may be attained as
the limit of some sequence of residues all contained in a common
apartment. The third section is devoted to the horofunction
compactification and proves that in the case of $\Res(X)$ the
combinatorial and horofunction compactifications coincide. Chabauty
topology is studied in the next section, whose main goal is to prove
Theorem~\ref{chab:intro}. The next section studies the relationship
between the visual boundary and the combinatorial compactification.
The main results are a \textbf{stratification} of the combinatorial
compactification (Theorem~\ref{thm:stratification}) and a
description of $\Cr(X)$ as the quotient of the \textbf{refined
visual boundary} of $X$ which is a refined version of the visual
boundary introduced in~\cite{Cap} for arbitrary \cat
spaces. These results are used in the final section which proves
Theorem~\ref{amenable:intro}. Finally, the Appendix outlines similar
results in the case of finite-dimensional \cat cube complexes.

\subsection*{Acknowledgment}

The second author is very grateful to Bertrand R\'emy for his constant support.

\section{The root-distance on spherical
residues}\label{sec:root-distance}

\subsection{Preliminaries}
Throughout this paper we let $X$ be an arbitrary building of finite
rank and $G$ be its full automorphism group. We denote by $\ch(X)$
(resp. $\cl(X)$, $\Res(X)$) the set of chambers (resp. panels,
spherical residues) of $X$. Given a residue $\sigma$ of $X$, the
\textbf{star} of $\sigma$, denoted by $\St(\sigma)$, is the set of
all residues containing $\sigma$ in their boundaries, see
\cite[\S1.1]{Tits74}. We recall that, in the chamber system approach to
buildings, which is dual to the simplicial approach, a residue is
viewed as a set of chambers and the star is then nothing but the set
of all residues \emph{contained in} $\sigma$. This has no influence
on the subsequent considerations and the reader should feel free to
adopt the point of view which he/she is most comfortable with.

\subsection{The root-distance}

Our first task is to introduce a combinatorial distance on the set $\Res(X)$ of spherical residues. A natural
way to obtain such a distance is by considering the \textbf{incidence graph} of spherical residues, namely the
graph with vertex set $\Res(X)$ where two residues are declared to be adjacent if one is contained in the other.
However, the disadvantage of this graph is that the natural embedding of $\ch(X)$ in $\Res(X)$ is not isometric, when
$\ch(X)$ is endowed with the gallery distance. This causes some technical difficulties which we shall avoid by
introducing an alternative distance on $\Res(X)$.

Given $R_{1}, R_{2}\in \Res(X)$, let $A$ be an apartment containing them both. We denote by $\Phi_{A}(R_{1},
R_{2})$ the set of all half-apartments of $A$ containing $R_{1}$ but not $R_{2}$. This set is empty if and only
if $R_{1}$ is contained in $R_{2}$, since every residue coincides with the intersection of all half-apartments
containing it.  Notice moreover that the cardinality of the sets $\Phi_{A}(R_{1}, R_{2})$ and  $\Phi_{A}(R_{2},
R_{1})$  is independent of the choice of $A$. We define the \textbf{root-distance} $d(R_{1},R_{2})$ between
$R_{1}$ and $R_{2}$ to be half of the sum of their cardinalities. In symbols:
$$d(R_{1}, R_{2}) = \frac 1 2 |\Phi_{A}(R_{1}, R_{2})| + \frac 1 2 | \Phi_{A}(R_{2}, R_{1}) |.$$

Clearly the restriction of the root-distance to the chamber set coincides with the gallery distance. However,
checking that the root-distance indeed defines a metric on $\Res(X)$ requires some argument (see
Proposition~\ref{prop:RootDistance}). Before collecting this together with some other basic facts on the root
distance, we introduce some additional terminology.

A set of spherical residues $\mathcal R \subset \Res(X)$ is called \textbf{closed} if for all $R_{1}, R_{2} \in \Res(X)$, we have
$$R_{1} \subset R_{2} \in \mathcal R \hspace{.2cm}\Rightarrow \hspace{.2cm} R_{1}  \in \mathcal R.$$
It is called \textbf{convex} if it is closed and if for all $R_{1}, R_{2} \in \mathcal R$, we have
$\proj_{R_{1}}(R_{2}) \in \mathcal R$, where $\proj$ denotes the \textbf{combinatorial projection} (see
\cite[\S3.19]{Tits74}) or \cite[Section~4.9]{AB}). We recall that by definition we have 
$$\proj_{R_{1}}(R_{2}) = \bigcap \{ \proj_{R_1}(C) \; | \; C \in  \ch(X) \cap \St(R_2) \},$$
which allows one to recover the combinatorial projections amongst \emph{arbitrary residues} from projections of \emph{chambers}. 

Since any intersection of closed (resp. convex) subsets is closed (resp.
convex) and since the whole set $\Res(X)$ is so, it makes sense to consider the \textbf{closure} (resp. the
\textbf{convex hull}) of a subset $\mathcal R \subset \Res(X)$, which we denote by $\overline{\mathcal R}$
(resp. $\Conv(\mathcal R)$). The convex hull of two residues $R_1, R_2$ is denoted by $\Conv(R_1,R_2)$. Given an
apartment $A$ containing $R_1 \cup R_2$, the convex hull $\Conv(R_1, R_2)$ coincides with the intersection of
all half-apartments of $A$ containing $R_1 \cup R_2$. The following basic fact provides a convenient
characterisation of the combinatorial projection:

\begin{lemma}\label{lem:proj-conv}
Given two (spherical) residues $R, T$, the combinatorial projection $\proj_R(T)$ coincides with the unique
maximal residue containing $R$ and contained in $\Conv(R, T)$.
\end{lemma}

`Maximal' should be understood as `of maximal possible dimension', \emph{i.e.} of minimal possible rank.

\begin{proof}
Let $A$ be an apartment containing $\Conv(R, T)$ and $\sigma_1, \sigma_2 \in \Conv(R, T)$ be two maximal
residues containing $R$. Assume $\sigma_1$ and $\sigma_2$ are distinct. Then there is a half-apartment $\alpha$
of $A$ containing one but not the other. Without loss of generality $\sigma_1 \subset \alpha$ but $\sigma_2 \not
\subset \alpha$. Since $R \subset \sigma_1 \cap \sigma_2$ we have $R \subset \partial \alpha$. Therefore, if $T
\subset \alpha$, then $\Conv(R, T) \subset \alpha$ which contradicts $\sigma_2 \not \subset \alpha$. Thus $T$
meets in the interior of $-\alpha$. In particular, so does $\proj_{\sigma_1}(T)$. Since the latter is a spherical
residue containing $\sigma_1 \supset R$, we have $\sigma_1 =\proj_{\sigma_1}(T)$, which contradicts the fact
that $\sigma_1 $ is contained in $\alpha$.

This confirms that there is a unique maximal residue $\sigma \in \Conv(R, T)$ containing $R$. Since $\proj_R(T)
\supset R$, we have thus $\proj_R(T) \subset \sigma$. If the latter inclusion were proper, then there would
exist some root $\beta$ containing $\proj_R(T)$ but not $\sigma$. In particular $R$ and $\proj_R(T)$ are
contained in the wall $\partial \beta$. This implies that $T$ is also contained in $\partial \beta$. Therefore
so is $\Conv(R, T)$ since walls are convex. This contradicts $\sigma \not \subset \beta$.
\end{proof}

We next introduce the \textbf{interval} determined by two spherical residues $R_{1}, R_{2}$ as the set $[R_{1},
R_{2}]$ consisting of those $\sigma \in \Res(X)$ such that $d(R_{1}, R_{2}) = d(R_{1}, \sigma)  + d(\sigma,
R_{2})$.


\begin{proposition}\label{prop:RootDistance}
We have the following.
\begin{itemize}
\item[(i)]  The root-distance turns the set $\Res(X)$ into a (discrete) metric space.

\item[(ii)] Retractions on apartments do not increase the root-distance.

\item[(iii)] For all $R_1, R_2 \in \Res(X)$, we have $\Conv(R_1, R_2) = \overline{[R_1, R_2]}$.

\item[(iv)] A set $\mathcal R \subset \Res(X)$ is convex if and only if it is closed and for all $R_1, R_2 \in
\mathcal R$, the interval $[R_1, R_2]$ is entirely contained in $\mathcal R$.
\end{itemize}
\end{proposition}

Before undertaking the proof, we record the following subsidiary fact which will be helpful in many arguments
using induction on the root-distance.

\begin{lemma}\label{lem:interval}
Let $R_1, R_2 \in \Res(X)$. Then the interval $[R_1, R_2]$ coincides with the pair $\{R_1, R_2\}$ if and only if
$R_1 \subset R_2$ or $R_2 \subset R_1$ and no residue other than $R_1$ or $R_2$ is sandwiched between them.
\end{lemma}

\begin{proof}
The `if' part is straightforward. Moreover, if $R_1 \subset R_2$ and $R$ is a residue with $R_1 \subset R
\subset R_2$, then $R \in [R_1, R_2]$. Therefore, it suffices to show that if $R_1 \cap R_2$ is different from
$R_1$ or $R_2$, then $]R_1, R_2[{}  := [R_1, R_2] \setminus \{R_1, R_2\}$ is non-empty.

Consider the \cat realisation $|X|$ of $X$ (see \cite{Dav98}). Recall that the \textbf{support} of a point $x
\in |X|$ is the unique minimal (\emph{i.e.} lowest dimensional) spherical residue $R$ such that $x \in |R|$.

\smallskip
Assume first that there exist points $p_1 \in |R_2|$ and $p_2 \in |R_2|$ such that $p_i$ is supported
by $R_i$ and that the geodesic segment $[p_1, p_2]$ is not entirely contained in $|R_1| \cup | R_2|$. Let then
$x$ be a point of $[p_1, p_2] \setminus (|R_1| \cup | R_2|)$ and let $R$ denote the spherical residue supporting
$x$. Clearly $R \neq R_1, R_2$. We claim that $R \in [R_1, R_2]$.

Let $A$ be an apartment containing $R_1$ and $R_2$. Then $R \subset A$. Since any root either contains $R$ or
does not, we have $\Phi_A(R_1, R_2) \subset \Phi_A(R_1, R) \cup \Phi_A(R, R_2)$ and similarly with $R_1$ and
$R_2$ interchanged. Thus it suffices to show that every root $\alpha$ containing $R$ but not $R_2$ also contains
$R_1$ and vice-versa. But if $\alpha$ does not contain $R_1$, it does not contain $p_1$ since $p_1$ lies in the
interior of $R_1$. Thus the wall $\partial \alpha$ does not separate $p_1$ from $p_2$, which contradicts the
fact that $x \in |R| \subset |\alpha|$. This proves the claim.

\smallskip
Assume in a second case that for all points $p_1, p_2$ respectively supported by $R_1,R_2$, the geodesic segment
$[p_1, p_2]$ lies entirely in $|R_1| \cup |R_2|$. Then $|R_1| \cap |R_2|$ is non-empty and $R:= R_1 \cap R_2$ is
thus a non-empty spherical residue. By the above, the residue $R$ is different from $R_1$ and $R_2$. We claim
that $R \in [R_1, R_2]$. Let $A$ be an apartment containing $R_1$ and $R_2$; thus $R \subset A$. As before, it
suffices to show that every root $\alpha $ of $A$ containing $R$ but not $R_1$ also contains $R_2$. If it
didn't, then $-\alpha$ would contain two points $p_1, p_2$ respectively supported by $R_1, R_2$. In particular
the geodesic segment $[p_1, p_2]$ is entirely contained in the interior of $-\alpha$ and, hence, it avoids $|R|
\subset \alpha$. This is absurd since $|R| = |R_1| \cap |R_2|$.
\end{proof}

\begin{proof}[Proof of Proposition~\ref{prop:RootDistance}]
We start with the proof of (ii). Let $\rho$ be a retraction to some apartment $A$ and let $R_1, R_2 \in
\Res(X)$. We need to show that $d(\rho(R_1), \rho(R_2)) \leq d(R_1, R_2)$. We work by induction on $d(R_1,
R_2)$, the result being trivial if $R_1 = R_2$. Notice more generally that if $R_1 \subset R_2$, then the
restriction of $\rho$ to the pair $\{R_1, R_2\}$ is isometric, in which case the desired inequality holds
trivially. We may therefore assume that $R_1$ and $R_2$ are not containing in one another. By
Lemma~\ref{lem:interval}, this implies that the open interval $]R_1, R_2[$ is non-empty. Let $R \in {}]R_1, R_2[$.
Using the induction hypothesis, we deduce
$$\begin{array}{rcl}
d(R_1, R_2) &=& d(R_1, R) + d(R, R_2)\\
& \geq & d(\rho(R_1), \rho(R)) + d(\rho(R), \rho(R_2))\\
& \geq & d(\rho(R_1), \rho(R_2)),
\end{array}$$
where the last inequality follows since any root of $A$ either contains $R$ or does not, whence $\Phi_A(R_1,
R_2) \subset \Phi_A(R_1, R) \cup \Phi_A(R, R_2)$ and similarly with $R_1$ and $R_2$ interchanged.

\medskip \noindent
(i) The only non-trivial point to check is the triangle inequality. We have just observed along the way that
this inequality holds for triple of residues contained in a common apartment. The general case follows, using
the fact that retractions do not increase distances.

\medskip \noindent
(iii) We first use induction on $d(R_1, R_2)$ to prove that $[R_1, R_2] \subset \Conv(R_1, R_2)$.

Let thus $R \in [R_1, R_2]$. We shall show by induction on $d(R, R_2)$ that $R \in \Conv(R_1,R_2)$.

Assume first that $]R, R_2[$ contains some spherical residue $T$. Then
$$R \in [R_1, T] \subset \Conv(R_1, T) \subset \Conv(R_1, R_2),$$
where the first inclusion follows from the induction on $d(R_1, R_2)$ and the second from the induction on $d(R,
R_2) > d(T, R_2)$.

Assume now that $]R, R_2[$ is empty. If $R \subset R_2$, then obviously $R \in \Conv(R_1, R_2)$. In view of
Lemma~\ref{lem:interval} it only remains to deal with the case $R_2 \subsetneq R$. In particular $d(R_1, R) <
d(R_1, R_2)$, whence $[R_1, R] \subset \Conv(R_1, R)$ by induction. Since $\Conv(R_1, R)$ is closed, it contains
$R_2$ and we deduce that some apartment $A$ contains $R_1 \cup R_2 \cup R$. Finally we observe that $\Conv(R_1,
R)= \Conv(R_1, R_2)$, since the fact that $R \in [R_1, R_2]$ implies that any root of $A$ which contains $R$ but
not $R_2$ also contains $R_1$. Thus $R \in \Conv(R_1, R_2)$, which confirms the claim that $[R_1, R_2] \subset
\Conv(R_1, R_2)$. In particular $\overline{[R_1, R_2]} \subset \Conv(R_1, R_2)$ since convex sets are closed.

\smallskip Let now $x \in \Conv(R_1, R_2)$ and pick a maximal spherical residue $R \in \Conv(R_1, R_2)$ containing
$x$. We claim that $R \in [R_1, R_2]$. Let thus $\alpha$ be a root containing $R$ but neither $R_2$ in some
apartment $A$ containing $R_1 \cup R_2$. If $R_1 \not \subset \alpha$, then $\Conv(R_1, R_2) \subset -\alpha$
whence $R \subset \partial \alpha$. This implies that $\proj_R(R_2)$ is strictly contained in $-\alpha$, thereby
contradicting the maximality of $R$. This shows that every root containing $R$ but not $R_2$ also contains
$R_1$. A similar argument holds with $R_1$ and $R_2$ interchanged. This proves $R \in [R_1, R_2]$ as claimed.
Thus $x \in \overline{[R_1, R_2]}$, which finishes the proof of (iii).

\medskip \noindent
(iv) follows from (iii) since a set $\mathcal R \subset \Res(X)$ is convex if and only if it is closed and for
all $R_1, R_2 \in \mathcal R$, we have $\Conv(R_1, R_2) \subset \Res(X)$.
\end{proof}

The following shows that the combinatorial projection of residues is canonically determined by the root-distance. In the special case of projections of chambers, the corresponding statements are well known.

\begin{corollary}\label{cor:proj-interval}
For all $R, T \in \Res(X)$, the projection $\proj_{R}(T)$ coincides with the unique maximal element of $[R, T]$ which contains $R$. It is also the unique spherical residue $\pi \supset R$ such that
$$d(\pi, T) = \min \{d(\sigma, T) \; | \; \sigma \in \Res(X), \; \sigma \supset R\}.$$
\end{corollary}

\begin{proof}
By Proposition~\ref{prop:RootDistance}(iii), the projection $\proj_{R}(T)$ is contained in some spherical residue $\pi \in [R, T]$. In particular $\pi$ is contained in $\Conv(R, T)$ and contains $R$. Therefore we have  $\pi = \proj_{R}(T)$ by Lemma~\ref{lem:proj-conv}. Thus $\proj_{R}(T)$ is contained in the interval $[R, T]$ and the first assertion of the Corollary follows from Lemma~\ref{lem:proj-conv} since $[R, T] \subset \Conv(R, T)$ by Proposition~\ref{prop:RootDistance}(iii).

The second assertion follows from arguments in the same vein than those which have been used extensively in this section. The details are left to the reader.
\end{proof}

\section{Combinatorial compactifications}

\subsection{Definition}

The key ingredient for the construction of the combinatorial
compactifications is the \textbf{combinatorial projection}. Given a
residue $\sigma$, this projection is the map $\proj_\sigma : \ch(X)
\to \St(\sigma)$ which associates to a chamber $C$ the chamber of
$\St(\sigma)$ which is nearest to $C$, see \cite[\S4.9]{AB}. As recalled in the previous section, the combinatorial projection may be extended to a map defined
on the set of all residues of $X$. For our purposes, we shall focus
on spherical residues and view the combinatorial projection as a map
$$
\proj_\sigma : \Res(X) \to \St(\sigma).
$$

This allows one to define two maps
$$
\projc : \ch(X) \to \prod_{\sigma \in \cl(X)} \St(\sigma) : C
\mapsto \big(\sigma \mapsto \proj_\sigma(C)\big)
$$
and
$$
\projr : \Res(X) \to \prod_{\sigma \in \Res(X)} \St(\sigma): R
\mapsto \big(\sigma \mapsto \proj_\sigma(R)\big).
$$

The above products are endowed with the product topology, where each star is a discrete set of residues. This
allows one to consider the closure of the image of the above maps. In symbols, this yields the following
definitions:
$$
\Cc(X)  = \overline{\projc ( \ch(X)) } \hspace{1cm} \text{and} \hspace{1cm}  \Cr(X) =  \overline{\projr (
\Res(X)) }.
$$

It is quite natural to consider the space $\overline{\projr ( \ch(X)) }$ as well; in fact, we shall see in Proposition~\ref{prop:CcToCr} below that this is equivariantly homeomorphic to $\Cc(X)$. We shall also see that $\Cc(X)$ may be identified to a
closed subset of $\Cr(X)$.

If the building  $X$ is locally finite, then the star of each spherical residue is finite and hence the spaces
$\Cc(X)$ and $\Cr(X)$ are then compact, and even metrizable since $\Res(X)$ is at most countable. The $\Aut(X)$-action
on $X$ extends in a canonical way to actions on $\Cc(X)$ and $\Cr(X)$ by homeomorphisms; the action induced by an element $g \in Aut(X)$ is given by
$$g : \Cr(X) \to \Cr(X) : f \mapsto \big( \sigma \mapsto   gf(g^{-1}\sigma) \big).$$

\begin{definition}
The space $\Cc(X)$ and $\Cr(X)$ are respectively called the \textbf{minimal} and the \textbf{maximal
combinatorial bordifications} of $X$. When the building $X$ is locally finite, we shall use instead the term
\textbf{compactification}.
\end{definition}

This terminology is justified by the following.

\begin{proposition}\label{disccomb}
The maps $\projc$ and $\projr$ are $G$-equivariant and injective; moreover, they have discrete images. In
particular $\projc$ and $\projr$ are homeomorphisms onto their images.
\end{proposition}

\begin{proof}
We argue only with $\projc$, the case of $\projr$ being similar. The equivariance is immediate. We focus on the
injectivity. Let $C$ and $C'$ be distinct chambers in $X$. There exists an apartment, say $A$, containing them
both. These chambers are separated in $A$ by some wall $H$, so that the projections of $C$ and $C'$ on every
panel in $H$ cannot coincide. This implies that $\projc(C)\neq\projc(C')$ as desired.

Let now  $(C_n)_{n\geq 0}$ be a sequence of chambers such that the sequence $(\projc(C_n))$ converges to
$\projc(C)$ for some $C \in \ch(X)$. We have to show that $C_n=C$ for $n$ large enough. Suppose this is not the
case. Upon extracting a subsequence, we may assume that $C_n\neq C$ for all $n$. Then there is some panel
$\sigma_n$ in the boundary of $C$ such that $\proj_{\sigma_n}(C_n)\neq C$. Up to a further extraction, we may
assume that $\sigma_n$ is independent of $n$ and denote by $\sigma$ the common value. Thus, we have
$\proj_\sigma(C_n)\neq C$, which contradicts the fact that  $(\projc(C_n))$ converges to $\projc(C)$.
\end{proof}

In the case when $X$ is locally finite, Proposition~\ref{disccomb} implies that $\Cc(X)$ is indeed a
compactification of the set of chambers of $X$: In particular the discrete set $\projc(\ch(X))$ is open in
$\Cc(X)$, which is thus indeed a compactification of $\ch(X)$ in the locally finite case; a similar fact of
course holds for $\Cr(X)$.

The elements $\Cc(X)$ and $\Cr(X)$ are considered as functions which associate to every panel (resp. residue) a
chamber (resp. residue) in the star of that panel (resp. residue). In view of Proposition~\ref{disccomb} we may
-- and shall -- identify $\ch(X)$ and $\Res(X)$ to subsets of $\Cc(X)$ and $\Cr(X)$. In particular, it makes
sense to say that a sequence of chambers converges to a function in $\Cc(X)$.

\medskip%
We now take a closer look at the minimal bordification. The special case of a single apartment is
straightforward:

\begin{lemma}\label{remarque}
Let $f\in \Cr(X)$ and  let $(R_n)$ and $(T_n)$ be sequences of spherical residues in a common apartment $A$ such that
$(\projr(R_n))$ and $(\projr(T_n))$  both converge to $f$. Then for every half-apartment $\alpha$ in $A$, there
is some $N\in\N$ such that either $R_n \cup T_n \subset \alpha$  for all $n>N$, or $R_n \cup T_n \subset
-\alpha$ for all $n>N$, or  $R_n \cup T_n \subset \partial \alpha$  for all $n>N$.

Conversely, let $(R_n)$ be a sequence of spherical residues of $A$ such that for every half-apartment $\alpha$ in $A$,
there is some $N\in\N$ such that either  $R_n \cup T_n \subset \alpha$  for all $n>N$, or $R_n \cup T_n \subset
-\alpha$ for all $n>N$, or  $R_n \cup T_n \subset \partial \alpha$  for all $n>N$.
Then $(R_n)$ converges in $\Cr(A)$.

\smallskip The same statements hold for  sequences of chambers of $A$ and a point $f\in \Cc(X)$.
\end{lemma}

\begin{proof}
We deal only with the maximal bordification, the case of $\Cc(X)$ being similar but easier.

Let $H$ be a wall of $A$, let $\sigma \subset H$ and $C,C'$ be the two chambers of $A$ containing $\sigma$. For
any spherical residue $R$ of $A$, the projection $\proj_{\sigma}(R)$ coincides with $C$ (resp. $C'$) if and only
if $R$ lies on the same side of $H$ as $C$ (resp. $C'$). It coincides with $\sigma$ itself if and only if $R$
lies on $H$. The result now follows from the very definition of the convergence in $\Cr(X)$.

Let conversely $(R_n)$ be a sequence of spherical residues of $A$ which eventually remain on one side of every
wall of $A$. Let $R \in \Res(A)$. Let $\Phi$ denote the set of all roots $\alpha$ such that $R \subset \partial
\alpha$ and $(R_n)$ eventually penetrates and remains in $\alpha$. Since $\Phi$ is finite, there is some $N$
such that $R_n \subset \bigcap_{\alpha \in \Phi}\alpha$ for all $n > N$. In particular $\proj_R(R_n) \subset
\bigcap_{\alpha \in \Phi}\alpha$ for all $n > N$. It follows from Lemma~\ref{lem:proj-conv} that $\proj_R(R_n)$
coincides with the unique maximal spherical residue contained in $\bigcap_{\alpha \in \Phi}\alpha$ and
containing $R$. In particular, this is independent of $n > N$. Thus $(R_n)$ indeed converges in $\Cr(A)$.
\end{proof}

The subset of $\Cc(X)$ consisting of limits of sequences of chambers of an apartment $A$ is denoted by $\Cc(A)$,
and $\Cr(A)$ is defined analogously. One verifies easily that this  is consistent with the fact that $\Cc(A)$
(resp. $\Cr(A)$) also denotes the minimal (resp. maximal) bordification of the thin building $A$. However, it is
not clear \emph{a priori} that for every $f\in \Cc(X)$ belongs to $\Cc(A)$ for some apartment. Nevertheless, it
turns out that this is indeed the case:

\begin{proposition}\label{prop:phi}
The set $\Cc(X)$ is the union of $\Cc(A)$ taken over all apartments $A$. Similarly $\Cr(X)$ is covered by the
union of $\Cr(A)$ over all apartments $A$.
\end{proposition}

This proposition is crucial in understanding the combinatorial compactifications, since it allows one to reduce
many problems to the thin case. The proof requires some preparation and is thus postponed to \S\ref{sec:sectors}


\begin{example}\label{exarbres}
Trees without leaves are buildings of type $D_\infty$. Panels in these trees are vertices and a sequence of
chambers ({\it i.e.} edges) $(x_n)$ converges in the minimal bordification if the projection of $x_n$ on every vertex is eventually constant.
It is easy to check that $\Cc(X)$ is isomorphic to the usual bordification of the tree, that is, to its set of
ends.

It is also possible to view a homogeneous tree of valency $r\geq 1$ as the Coxeter complex associated to the group $W=\langle s_1,\dots,s_r\,|\,s_1^2,\dots,s_r^2\rangle$. The panels are then the middles of the edges, a chamber is a vertex with all the half-edges which are incident to it. From this viewpoint as well, the combinatorial bordification coincides with  the visual one.
\end{example}

\begin{example} (This example may be compared to \cite[6.3.1]{GuR})\label{exA2}
Let us consider an apartment $A$ of type $\tilde A_2$. It is a Euclidean plane,
 tessellated by regular triangles. We know by Lemma~\ref{remarque} that we can characterize the points $\xi\in \Cc(A)$ by the sets of roots $\Phi(\xi)$ associated to them. We may distinguish several types of boundary points. Let us choose some root basis $\{a_1,a_2\}$ in the vectorial system of roots. Then there is a point $\xi\in \Cc(A)$ defined by $\Phi(\xi)=\{a_1+k,a_2+l|\,k,l\in\Z\}$. There are six such points, which correspond to a choice of positive roots, {\it i.e.} to a Weyl chamber in $A$. The sequences of (affine) chambers that converge to these points are the sequences that eventually stay in a given sector, but whose distances to each of the two walls in the boundary of this sector tend to infinity.

There is also another category of boundary points, which corresponds to sequence of chambers that stay in a given sector, but stay at bounded distance of one of the two walls defining this sector. With a choice of $a_1$ and $a_2$ as before, these are points associated to set of roots of the form $\{a_1+k,a_2+l|k,l\in\Z,k\leq k_0\}$. As $k_0$ varies, we get a `line' of such points, and there are six such lines.

When $X$ is a building of type $\tilde A_2$, by Proposition~\ref{prop:phi}, we can always write a point in the
boundary of $X$ as a point in the boundary of some apartment of $X$. Thus, the previous description applies to general points of the bordification.
\end{example}

\begin{example}
Let $W$ be a Fuchsian Coxeter group, that is, whose Coxeter complex is a tessellation of the hyperbolic plane. Assume the action of $W$ on the hyperbolic plane is cocompact. As in the previous example we shall content ourselves with a description of  the combinatorial compactification of some apartment $A$. In order to do so, we shall use the visual boundary $\partial_\infty\Sigma\simeq S^1$. If there is a point $\xi$ of this boundary towards which no wall is pointing, then we can associate to it a point of the combinatorial compactification, just by taking the roots that contain a sequence of points converging to $\xi$.  If we have a point of the boundary towards which $n$ walls are pointing, we associate to it $n+1$ points in $\Cc(A)$, whose positions are defined in relations to the roots which have these $n$ walls as a boundary.

Moreover, let us remark that the set of limit points of walls is dense in the boundary of the hyperbolic plane. To prove that, it is enough to check that the action of $W$ on $S^1$ is minimal (all its orbits are dense). Using \cite[Corollary 26]{GdlH}, the action of $W$ on its limit set $L(W)$ is minimal, and by \cite[Theorem 4.5.2]{Ka}, the limit set $L(W)$ is in fact $S^1$.

Therefore, if $\xi$ and $\xi'$ are two regular points ({\it i.e.} towards which no wall is pointing), then $\xi$ and $\xi'$ are separated by some wall. In particular, we see that the construction we have just made always yields different points. Therefore, the compactification $\Cc(A)$ is a refinement of the usual boundary.
\end{example}

\begin{example}\label{exproduit}
If $X=X_1\times X_2$ is a product of buildings, then a chamber of $X$ is just the product of a chamber in $X_1$
and a chamber in $X_2$, and the projection on another chamber is just the product of the projections in $X_1$
and $X_2$. Therefore, the combinatorial compactification $\Cr(X)$ is the product $\Cr(X_1)\times \Cr(X_2)$.
\end{example}
\vspace{0.5 cm}

\subsection{Projecting from infinity}

In this section, we show that any function  $f \in \Cc(X)$ admits a projection on every spherical residue
of $X$. This allows us to define the embedding of $\Cc(X)$ into $\Cr(X)$ alluded to above.

Let thus $\xi \in \Cc(X)$ and $R$ be any residue. Recall
that $R$ is a building \cite[Theorem~3.5]{Ro}. We define the \textbf{projection} of $\xi$ to $R$, denoted  $\proj_R(\xi)$, to be the restriction of $\xi$ to $\cl(R)$.  In the special case when $R$ is a panel, the set $\cl(R)$ is a singleton and the function $\proj_R(\xi)$ may therefore be identified with a chamber which coincides with $\xi(R)$.

Similarly, given  $\xi \in \Cr(X)$ we define $\proj_R(\xi)$ as the restriction of $\xi$ to $\Res(R)$.

The next statement ensures that the definition of $\proj_R$ is meaningful.

\begin{lemma}\label{proj}
Let $(C_n)$ be a sequence of chambers converging to $\xi\in \Cc(X)$
and let $R$ be a residue in $X$. The sequence of projections
$(\proj_R(C_n))$ converges to an element $\proj_R(\xi)\in
\prod_{\sigma \in \cl(R)}\St(\sigma)$. In particular $\proj_R(\xi)$
is an element of $\Cc(R)$.

Similarly, any sequence  $(R_n)$  converging to some $\eta \in \Cr(X)$ yields a sequence $(\proj_R(R_n))$ which converges to some element of $\Cr(R)$ which is denoted by $\proj_R(\eta)$.
\end{lemma}
\begin{proof}
We focus on the minimal compactification; the maximal one is similar.

It is enough to prove the very first point. By definition of the convergence in $\Cc(X)$, for every panel
$\sigma \supset R$, there exists some integer $N$ depending on $\sigma$ such that for $n>N$,
$\proj_\sigma(C_n)=\xi(\sigma)$. Moreover we have
$$\proj_R(\proj_\sigma(C_n))=\proj_R(\xi(\sigma))=\xi(\sigma).$$

Now $\proj_\sigma(C_n)$ coincides with $\proj_\sigma(\proj_R(C_n))$. Hence, for $n>N$, we have
$$\proj_\sigma(\proj_R(C_n))=\xi(\sigma),$$
which is equivalent to saying that $(\proj_R(C_n))$ converges to $\proj_R(\xi)$.
\end{proof}

In Lemma~\ref{proj}, if the residue $R$ is spherical then for any
$\xi \in \Cc(X)$, the projection $\proj_{R}(\xi)$ may be identified
with a chamber of $R$. Similarly, for any  $\eta \in \Cr(X)$, the
projection $\proj_{R}(\eta)$ may be identified with a residue in
$\St(R)$.

\medskip
Let now $C$ be a chamber and $\xi$ be a point in the boundary of $\Cc(X)$. As $\xi$ is not equal to $C$, there
exists some panel in the boundary of $C$ on which the projection of $\xi$ is different from the projection of
$C$. Let $I$ be the set of all such panels, and consider the residue $R$ of type $I$ containing $C$.

\begin{lemma}\label{spher}
The residue $R$ is spherical and $\proj_R(\xi)$ is a chamber opposite $C$ in $R$.
\end{lemma}

\begin{proof}
Let $(C_n)$ be a sequence converging to $\xi$. By Lemma \ref{proj}, $(\proj_R(C_n))$ converges to $\proj_R(\xi)$. Therefore, the projection of $\proj_R(C_n)$ on the panels adjacent to $C_n$ is eventually the same as the projection of  $\proj_R(\xi)$, that is, of $\xi$. In other words, the projection of $C_n$ on every panel of $R$ in the boundary of $C$ is always different from $C$. By \cite[IV.6, Lemma 3]{Bro}, this implies that $R$ is spherical and that $\proj_R(C_n)$ is opposite to $C$. As $\proj_R(C_n)$ converges to $\proj_R(\xi)$, this implies that $\proj_R(\xi)$ is opposite to $C$ in $R$.
\end{proof}

\begin{definition}\label{defspher}
The residue $R$ defined as above is called the \textbf{residual projection} of $\xi$ on~$C$.
\end{definition}

\begin{proposition}\label{prop:CcToCr}
There is a $G$-equivariant continuous injective map $\Cc(X) \to \Cr(X)$. Its image coincides with the closure of $\ch(X)$ in $\Cr(X)$.
\end{proposition}

\begin{proof}
As pointed out before, the set $\Res(X)$ may be identified with a subset of $\Cr(X)$ via the map $\projr$. In particular we may view $\ch(X)$ as a subset of $\Cr(X)$. Projections to residues allows one to extend this inclusion to a well defined map $\Cc(X) \to \Cr(X)$. The fact that it is injective and continuous is straightforward to check; the details are left to the reader.
\end{proof}

In view of this Proposition, we may identify $\Cc(X)$ to a closed subset of $\Cr(X)$. The fact that $\Cc(X) \cap \Res(X)$ coincides with $\ch(X)$ motivates the following definition.

\begin{definition}\label{def:chamber}
A point of $\Cr(X)$ which belongs to $\Cc(X)$ is called a \textbf{chamber}. If it does not belong to $\ch(X)$, we say that it is a \textbf{chamber at infinity}.
\end{definition}

\subsection{Extending the notion of sectors to arbitrary buildings}
The notion of {sectors} is crucial in analysing the structure of Euclidean buildings. In this section we propose
a generalisation of this notion to arbitrary buildings. This will turn out to be a crucial tool for the study of
the combinatorial bordifications.

Let $x \in \Res(X)$ be a spherical residue and $(R_{n})$ be a sequence of spherical residues converging to some
$\xi \in \Cr(X)$. In order to simplify the notation, we shall denote the sequence $(R_n)$ by $\underline R$. For
any integer $k \geq 0$ we set
$$Q_k=\bigcap_{n\geq k}{\mathrm {Conv}}(x,R_n)$$
and
$$Q(x,\underline R)=\bigcup_{k\geq 0}Q_k.$$ 

Since $Q_{k}$ is contained in an apartment and since $Q_{k} \subset Q_{k+1}$ for all $k$, it follows from
standard arguments that $Q(x, \underline R)$ is contained in some apartment of $X$ (compare \cite[\S3.7.4]{Tit}
or \cite[Theorem 3.6]{Ro}).

\begin{remark}\label{rem:sectors}
Retain the same notation as before and  let $y$ be a spherical residue. If $y\subset  Q(x,\underline R)$, then we have $Q(y,\underline R)\subset Q(x,\underline R)$.
\end{remark}

\begin{proposition}\label{Qxi}
Let $(R_n)=\underline R$ be a sequence of spherical residues converging to $\xi\in \Cr(X)$ and let $x \in \Res(X)$.
\begin{enumerate}
\item[(i)]
The set $Q(x,\underline R)$ only depends on $x$ and $\xi$, and not on the sequence $\underline R$.

\item[(ii)] $Q(x,\underline R)$ may be characterised as the smallest subcomplex $P$ of $X$ containing $x$ and such that if $R$ is a spherical residue in $P$, then for every $\sigma \in \St(R)$, the projection $\proj_{\sigma}(\xi)$ is again in $P$.

\end{enumerate}
\end{proposition}

\begin{proof}
Clearly (i) is a consequence of (ii).

Set $Q := Q(x, \underline R)$  and define $Q'$ to be the \textbf{$\xi$-convex hull} of $x$. By definition, this
means that $Q'$ is the minimal set  of spherical simplices satisfying the following three conditions:
\begin{itemize}
\item $x\in Q'$.

\item $Q'$ is closed.

\item For any spherical residue $\sigma \subset Q'$  we have $\proj_\sigma(\xi)\subset Q'$.
\end{itemize}
We have to show that $Q =Q'$. To this end, let $\mathcal V$ denote the collection of all subsets of $\Res(X)$ satisfying the above three conditions. Thus $Q' = \bigcap \mathcal V$.

By definition, for each $k \geq 0$ the subcomplex $Q_{k}$ is convex, hence closed, and contains $x$. Therefore
the same holds true for $Q$. We claim that $Q \in \mathcal V$. Indeed, for any $\sigma \in \Res(X)$ the
projection  $\proj_{\sigma}(\xi)$ coincides with $\proj_{\sigma}(R_{n})$ for $n$ large enough (see
Lemma~\ref{proj}). Therefore, for any $\sigma \in Q$ there exists a sufficiently large $k$ such that
$\proj_{\sigma}(\xi) = \proj_{\sigma}(R_{n}) \subset \Conv(x, R_{n})$ for all $n >k$. Thus  $\proj_{\sigma}(\xi)
\subset Q_{k} \subset Q$, which confirms that $Q \in \mathcal V$. In particular we deduce  that $Q' \subset Q$.

\medskip
Let now $R $ be a spherical residue in $Q$. We shall show by induction on the root-distance of $R$ to $x$ that
$R \subset Q'$.

Assume first that $x \supset R$. Then $R \in Q'$ since $Q'$ is closed and contains $x$. Assume next that $x
\subset R$. As $R \in Q$, we have $R \in \Conv(x, R_n)$ for any $n$ large enough. Since $\proj_x(R_n)$ is the
largest residue contained in $\Conv(x, R_n)$ and containing $x$ (see Lemma~\ref{lem:proj-conv}), we have $R
\subset \proj_x(R_n)$. It follows that $R \in Q'$ since $Q'$ is closed and since for large $n$ we have
$\proj_x(R_n) = \proj_x(\xi) \in Q'$.

In view of Lemma~\ref{lem:interval}, we may now assume that the interval $[x, R]$ is non-empty and contains some
spherical residue $x'$. By induction $x' \in Q'$. Let $n$ be large enough so that $R \in \Conv(x, R_n)$. We have
$x' \in [x, R] \subset \Conv(x, R)$ by Proposition~\ref{prop:RootDistance}(iii). Thus, in an apartment
containing $\Conv(x, R_n)$, any root containing $x$ and $R$ also contains $x'$. One deduces that $R \in
\Conv(x', R_n)$ for all large $n$. In particular $R \in Q(x', \underline R)$. By induction, we deduce that $R$
belongs to the $\xi$-convex hull of $x'$, which we denote by $Q'(x')$. Since $x' \in Q'$, we have $Q'(x')
\subset Q'$ whence $R \in Q'$ as desired.
\end{proof}

Proposition~\ref{Qxi} shows that $Q(x, \underline R)$ depends only on $\xi = \lim \underline R$. It therefore makes sense to write  $Q(x,\xi)$ instead of $Q(x,\underline R)$.

\begin{definition}
The set $Q(x,\xi)$ is called the \textbf{combinatorial sector}, or simply the \textbf{sector}, pointing towards $\xi$ and based at $x$.
\end{definition}

\begin{remark}\label{rqQ}
In the affine case, sectors in the classical sense are also combinatorial sectors in the sense of the preceding
definition (see example \ref{exqA2}). However, the converse is not true in general.
\end{remark}

The following shows that the sector $Q(x, \xi) = Q(x, (R_n))$ should be thought of as the pointwise limit of
$\Conv(x, R_n)$ as $n$ tends to infinity:

\begin{corollary}\label{cor:sector:local}
Let $x \in \Res(X)$ and $(R_n)$ be a sequence of spherical residues converging to some $\xi \in \Cr(X)$. For any
finite subset $F \subset \Res(X)$ there is some $N\geq 0$ such that for any $n >N$, the respective intersections
of the sets $Q(x, \xi)$ and $\Conv(x, R_n)$ with $F$ coincide.
\end{corollary}

\begin{proof}
It suffices to show that for each $y \in F$, either $y \subset \Conv(x, R_n)$ for all large $n$, or $y \not
\subset \Conv(x, R_n)$ for all large $n$.

If $y \subset Q(x, \xi)$, this follows at once from the definition of $Q(x, \xi)$.

Assume now that $y$ is not contained in $Q(x, \xi)$. We claim that there is some $N
>0$ such that $y \not \subset \Conv(x, R_n)$ for all $n >N$. Indeed, in the contrary case for each $n>0$ there
is some $\varphi(n)>n$ such that $y$ is contained in $\Conv(x, R_{\varphi(n)})$. Therefore we have $y \subset
\bigcap_{n>0} \Conv(x,R_{\varphi(n)})$. Since $(R_{\varphi(n)})$ converges to $\xi$ it follows from
Proposition~\ref{Qxi} that $Q(x, (R_{\varphi(n)})) = Q(x, \xi)$ and we deduce that $y \subset Q(x, \xi)$, which
is absurd.
\end{proof}

The following interpretation of the projection from infinity shows in particular that for all $x \in \Res(X)$ and
$\xi \in \Cr(X)$, the projection $\proj_x(\xi)$ is canonically determined by the sector $Q(x, \xi)$, viewed as a
set of spherical residues.

\begin{corollary}\label{cor:proj}
For all $x \in \Res(X)$ and $\xi \in \Cr(X)$, the projection $\proj_x(\xi)$ coincides with the unique maximal
residue containing $x$ and contained in $Q(x, \xi)$.
\end{corollary}

\begin{proof}
Follows from Lemma~\ref{lem:proj-conv} and Corollary~\ref{cor:sector:local}.
\end{proof}

\begin{example}\label{exqarbre}
Let $X$ be a tree. It has been seen in example \ref{exarbres} that $\Cc(X)$ coincides with the visual compactification
$X\cup\partial_\infty X$. If $\xi\in\partial_\infty X$, then $Q(x,\xi)$ is the half-line starting from $x$ and
pointing towards $\xi$.

\end{example}
\begin{example}\label{exqA2}
Let $X$ be a building of type $\tilde A_2$. As explained in example \ref{exA2}, there are several type of boundary points. The sectors pointing towards different types of boundary points have different shapes. Furthermore, there are two possible orientations for chambers, which give also different shapes to the sectors.

Let $\xi_1$ be the point of the boundary which corresponds, in a given apartment $A$, to a set of roots $\Phi_A(\xi_1)=\{ a_1+k,a_2+l|k,l\in\Z\}$. Let $x$ be a chamber in $A$. To determine $Q(x,\xi_1)$, we have to consider the roots in the direction of $a_1$ and $a_2$ that contain $x$. We then have two possibilities for $Q(x,\xi_1)$, according to the orientation of $x$. If $x$ is oriented `towards $\xi_1$', then $Q(x,\xi_1)$ is the classical sector based at $x$ and pointing towards $\xi_1$. Otherwise, $Q(x,\xi_1)$ is a troncated sector. These two possibilities are described on Figure 1.

Now, let $\xi_2$ be the point in the boundary of $A$ determined by $\Phi_A(\xi_2)=\{a_1+k,a_2+l|k,l\in\Z,k\leq k_0\}$. The sector $Q(x,\xi_2)$ is determined in the same way as $Q(x,\xi_1)$, but we now have to stop at the root $a_1+k_0$. Thus, we get a `half-strip' which goes from $x$, stops at the root $a_1+k_0$, and is in the direction of $a_2$. Once again, the precise shape of this half-strip depends on the orientation of $x$. These two possibilites are described on Figure 2.

\begin{figure}[h]
   \begin{minipage}[c]{.46\linewidth}
\includegraphics*[120,110][320,300]{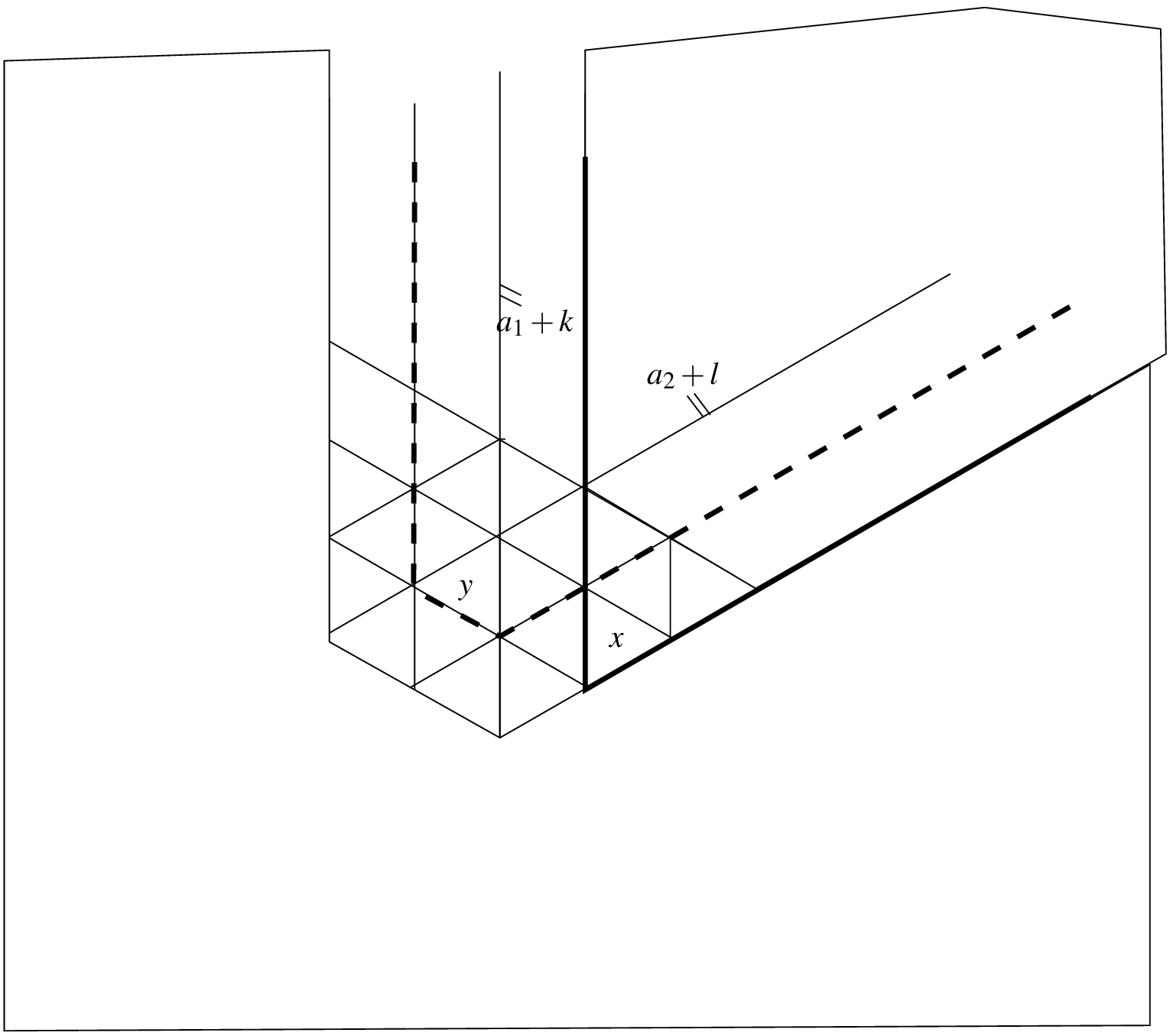}
\caption{$Q(x,\xi_1)$ and $Q(y,\xi_1)$, where $\Phi(\xi_1)=\{
a_1+k,a_2+l|k,l\in\Z\}$}
   \end{minipage} \hfill
   \begin{minipage}[c]{.46\linewidth}
\includegraphics*[180,130][375,350]{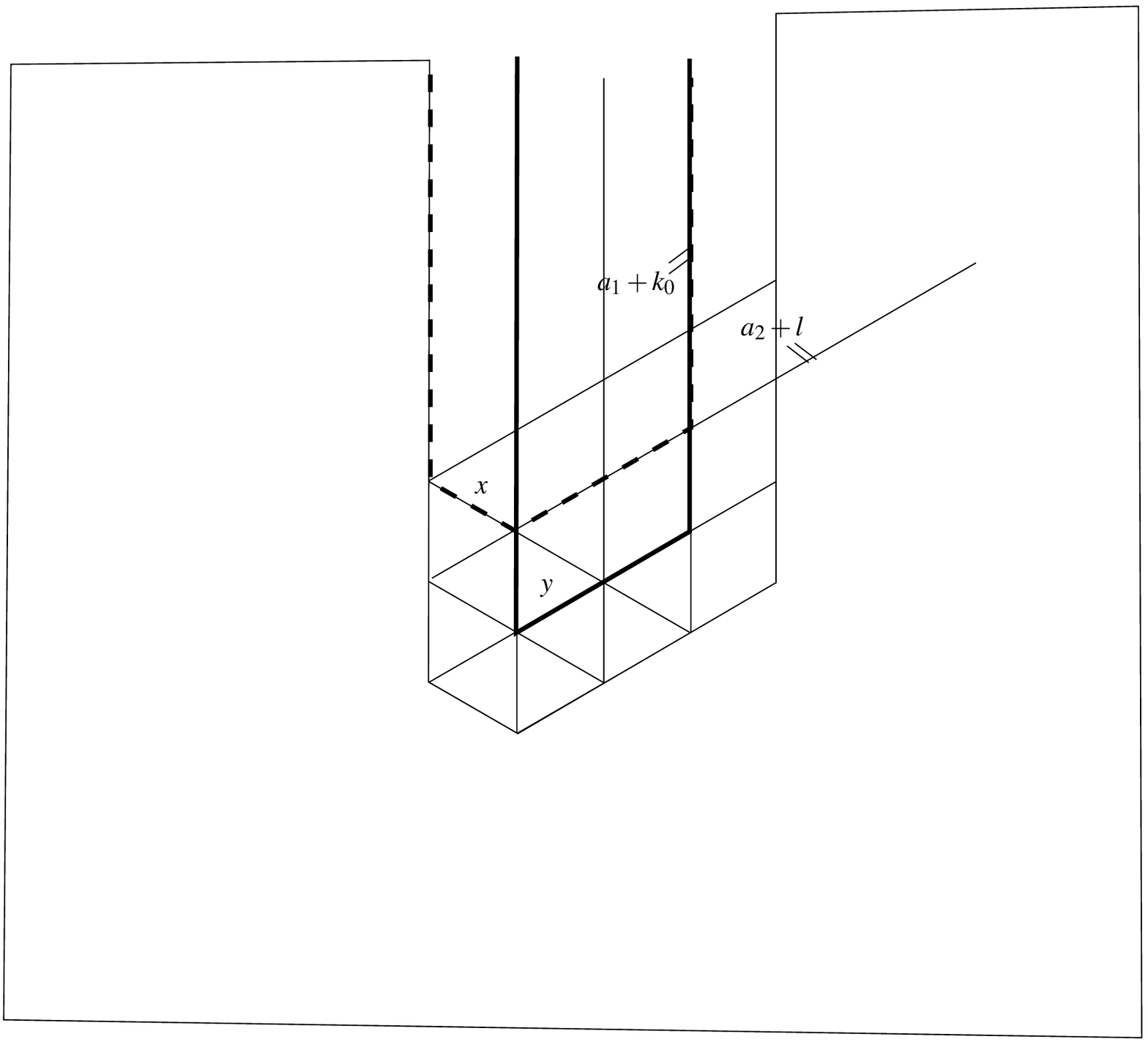}
\caption{$Q(x,\xi_2)$ and $Q(y,\xi_2)$, where
$\Phi(\xi_2)=\{a_1+k,a_2+l|k,l\in\Z,k\leq k_0\}$}
   \end{minipage}
\end{figure}

\end{example}

\subsection{Sectors and half-apartments}

The use of sectors will eventually allow us to study the combinatorial compactifications of $X$ by looking at
one apartment at a time. The first main goal is to obtain a proof of Proposition~\ref{prop:phi}. Not
surprisingly, retractions provide an important tool.

\begin{lemma}\label{lem:retraction:1}
Let $x \in \ch(X)$ and $(R_n)$ be a sequence of spherical residues converging to some $\xi \in \Cr(X)$. Let $A$
be an apartment containing the sector $Q(x, \xi)$. Then we have
$$Q(x, \xi)= \bigcup_{k \geq 0} \bigcap_{n \geq k} \Conv(x, \rho_{A, x}(R_n)).$$
\end{lemma}

\begin{proof}
Let $Q_k = \bigcap_{n \geq k} \Conv(x, R_n)$,  $Q'_k = \bigcap_{n \geq k} \Conv(x, \rho_{A, x}(R_n))$, $Q = Q(x,
\xi) = \bigcup Q_k$ and $Q' = \bigcup Q'_k$. We must show that $Q = Q'$. Since these are both (closed and)
convex and contain the chamber $x$, it suffices to show that $\ch(Q) = \ch(Q')$.

Fix $k\geq 0$. Let $C \in \ch(Q_k)$ and let $n \geq k$. Then $C$ belongs to a minimal gallery from $x$ to a
chamber containing $R_n$. Since $C \subset A$, the retraction $\rho_{A, x}$ fixes $C$ and hence $C$ belongs to a
minimal gallery from $x$ to a chamber containing $\rho_{A, x}(R_n)$. This shows that $\ch(Q_k) \subset
\ch(Q'_k)$. Therefore we have $\ch(Q) \subset \ch(Q')$.

Suppose for a contradiction that there exists some $C \in \ch(Q') \setminus \ch(Q)$. Choose $C$ at minimal
possible distance to $\ch(Q)$. Thus $C$ is adjacent to some chamber $C' \in \ch(Q)$. Let $\sigma = C \cap C'$ be
the panel shared by $C$ and $C'$ and $\alpha$ be the half-apartment containing $C'$ and such that $\partial
\alpha$ contains $\sigma$. Since $\sigma \subset Q$, we have $\proj_\sigma(x) \subset Q$, whence
$\proj_\sigma(x)= C'$ and $x \subset \alpha$.


Let now $k$ be such that $C \subset \Conv(x,\rho_{A, x}(R_n))$ for all $n \geq k$ and let $C_n$ denote the
unique element of $\Conv(x,R_n)$ such that $\rho_{A, x}(C_n)= C$. Each $C_n$ contains the panel $\sigma$ and we
have $\proj_\sigma(R_n)=C_n$. 
Therefore, we deduce that $\xi(\sigma)=C_n$ (see Lemma~\ref{proj}). Since $Q \subset A$, this implies that $C_n
\subset A$ by Proposition~\ref{Qxi}. Therefore we have  $C=\rho_{A, x}(C_n) = C_n$, whence $C \subset \Conv(x,
R_n)$ for all $n \geq k$. This implies that $C$ is contained in $Q$, which is absurd.
\end{proof}

\begin{lemma}\label{lem:retraction:2}
Let $x \in \ch(X)$ and $(R_n)$ be a sequence of spherical residues converging to some $\xi \in \Cr(X)$. Let $A$
be an apartment containing the sector $Q(x, \xi)$. Then for any chamber $C \in \ch(A)$, we have $Q(C, \xi)
\subset A$. Moreover there exists $k \geq 0$ such that $\rho_{A, C}(R_n) = \rho_{A,x}(R_n)$ for all $n >k$.
\end{lemma}

\begin{proof}
By connexity of $A$, it suffices to prove that for any chamber $C$ adjacent to $x$, the sector $Q(C, \xi)$ is
contained in $A$ and furthermore that $\rho_{A, C}(R_n) = \rho_{A,x}(R_n)$ for all sufficiently large $n$. Let
$\sigma = x \cap C$ be the panel separating $x$ from $C$.

Since $\sigma \subset Q(x, \xi)$, Proposition~\ref{Qxi} implies that $\xi(\sigma)$ is contained in $A$. The only
possible values for $\xi(\sigma)$ are thus $x, C$ and $\sigma$. We treat these three cases successively.

If $\xi(\sigma)=C$ then $C \subset Q(x, \xi)$, whence $Q(C, \xi) \subset A$ by Remark~\ref{rem:sectors}. The
desired claims follow by definition.

If $\xi(\sigma)=x$ then $x \in \Conv(C, R_n)$ for $n$ sufficiently large. Thus there is a minimal gallery from
$C$ to a chamber containing $R_n$ via $x$ and it follows that $\rho_{A, C}(R_n) = \rho_{A,x}(R_n)$ since $x \in
\ch(A)$. The fact that $Q(C, \xi)$ is contained in $A$ now follows from Lemma~\ref{lem:retraction:1}.

If $\xi(\sigma)=\sigma$, pick a large enough $n$ so that $\proj_\sigma(R_n)=\sigma$. Let $A_n$ be an apartment
containing $x \cup R_n$. Then $R_n$ lies on a wall $H_n$ of $A_n$ containing $\sigma$, and the convex hull
$\Conv(x, R_n)$ lies entirely on one side of $H_n$; we call the latter half-apartment $\alpha$. Since the
chamber $C$ contains the panel $\sigma \subset H_n$, there is a half-apartment $\beta$ bounding $H_n$ and
containing $C$. Therefore, upon replacing $A_n$ by $\alpha \cup \beta$, there is no loss of generality in
assuming that $C$ is contained in $A_n$. It follows readily that  $\rho_{A, C}(R_n) = \rho_{A,x}(R_n)$. As in
the previous case, the fact that $Q(C, \xi)$ is contained in $A$ follows from Lemma~\ref{lem:retraction:1}.
\end{proof}

The following is an analogue of Lemma~\ref{remarque} in the non-thin case.

\begin{lemma}\label{lem:roots}
Let $x \in \ch(X)$, $(R_n)$ be a sequence of spherical residues converging to some $\xi \in \Cr(X)$ and $A$ be
an apartment containing the sector $Q(x, \xi)$. Set $R'_n = \rho_{A,x}(R_n)$. For any half-apartment $\alpha$ of
$A$, there is some $N$ such that for all $n >N$ we have $R'_n \subset \alpha$ or $R'_n \subset -\alpha$.
\end{lemma}

\begin{proof}
Let $C, C'$ be the chambers of $A$ such that $C \cap C' = \sigma$. Let also $\alpha$ be the half-apartment
containing $C$ but not $C'$. Assume that the sequence $(R'_n)$ possesses two subsequences $R'_{\varphi(n)}$ and
$R'_{\psi(n)}$ such that $R'_{\varphi(n)}$ is strictly contained in $\alpha$ and $R'_{\psi(n)}$ is strictly
contained in $-\alpha$. Then $\Conv(C,R'_{\psi(n)})$ contains $C'$ for all $n$ while  $\Conv(C,R'_{\varphi(n)})$
does not. This contradicts Lemma~\ref{lem:retraction:1}, thereby showing that the sequence $R'_n$ eventually
remains on one side of the wall $\partial \alpha$.
\end{proof}

\begin{definition}
Let $\Phi_A(\xi)$ denote the set of all half-apartments $\alpha$ of $A$ such that the sequence $(R'_n)$
eventually lies in $\alpha$. In view of Lemma~\ref{lem:retraction:2}, this set is independent of $x \in \ch(A)$,
but depends only on $A$ and $\xi$.
\end{definition}

One should think of the elements of $\Phi_A(\xi)$ as half-apartments `containing' the point $\xi$. Notice that
if $\xi = R$ is a residue, then $\Phi_A(\xi)$ is nothing but the set of those half-apartments which contain $R$.

Notice that two opposite roots $\alpha, -\alpha$ might be both contained in $\Phi_A(\xi)$; in view of
Lemma~\ref{lem:roots}, this happens if and only if the residue $(R'_n)$ lies on the wall $\partial \alpha$ for
all sufficiently large $n$.

\begin{lemma}\label{lem:retraction:3}
Let $x \in \ch(X)$ and $(R_n)$ be a sequence of spherical residues converging to some $\xi \in \Cr(X)$. Let $A$
be an apartment containing the sector $Q(x, \xi)$. Then the sequence $(\rho_{A, x}(R_n))$ converges in $\Cr(A)$
and its limit coincides with the restriction of $\xi$ to $\Res(A)$.

Furthermore, for any $\xi' \in \Cr(A)$, we have $\xi' = \xi$ if and only if $\Phi_A(\xi') = \Phi_A(\xi)$.
\end{lemma}

\begin{proof}
Let $R \in \Res(A)$ and $H(R)$ denote the (finite) set of all walls containing $R$. By Lemma~\ref{lem:roots},
there is some $N$ such that $R'_n$ remains on one side of each wall in $H(R)$ for all $n >N$. The fact that  the
sequence $(R'_n)$ converges to some $\xi' \in \Cr(A)$ thus follows from Lemma~\ref{remarque}. By construction we
have $\Phi_A(\xi') = \Phi_A(\xi)$. All it remains to show is thus that $\xi $ and $\xi'$ coincide.

We first show that they coincide on panels. Let thus  $\sigma \subset A$ be a panel and $C,C' \in \ch(A)$ be
such that $C \cap C' = \sigma$ and $C \subset \alpha$. The following assertions are straightforward to check:
\begin{itemize}
\item $\xi(\sigma) = C$ if and only if $Q(C, \xi)$ and $Q(C', \xi)$ both contain $C$;

\item $\xi(\sigma) = C'$ if and only if $Q(C, \xi)$ and $Q(C', \xi)$ both contain $C'$;

\item $\xi(\sigma) = \sigma $ if and only if $Q(C, \xi)$ does not contain $C'$ and vice-versa.
\end{itemize}
Now remark that $Q(y, \xi) = Q(y, \xi')$ for any chamber $y \in \ch(A)$ in view of Lemma~\ref{lem:retraction:1}.
Therefore, we deduce from the above that $\xi$ and $\xi'$ coincide on $\sigma$.

It remains to observe that for any two spherical residues $R, R'$, the projection $\proj_R(R')$ is uniquely
determined by the set of all projections $\proj_\sigma(R')$ on panels $\sigma$ containing $R$. 
\end{proof}

The next result supplements the description provided by Proposition~\ref{Qxi}.

\begin{proposition}\label{prop:sector:roots}
Let $x \in \ch(X)$,  $\xi \in \Cr(X)$ and $A$ be an apartment containing the sector $Q(x, \xi)$. Then we have
$$Q(x, \xi) = \bigcap_{\alpha \in \Phi_A(x) \cap \Phi_A(\xi)} \alpha.$$
\end{proposition}

\begin{proof}
Set $Q =Q(x, \xi)$  and $Q'= \bigcap_{\alpha \in \Phi_A(x) \cap \Phi_A(\xi)} \alpha$. By
Lemma~\ref{lem:retraction:3} there is a sequence $(R'_n)$ of spherical residues of $A$ converging to $\xi$ in
$\Cr(A)$.

A half-apartment $\alpha$ belongs to $\Phi_A(\xi)$ if and only if $R'_n \subset \alpha$ for any sufficiently
large $n$. Therefore, the equality $Q = Q'$ follows from Lemma~\ref{lem:retraction:1}.
\end{proof}

The following result allows one to extend all the results of this section to sectors based at any spherical
residues, and not only at chambers. It shows in particular that sectors based at residues which are not chamber
may be thought of as \textbf{sector-faces}:

\begin{corollary}\label{cor:SectorFace}
Let $x \in \Res(X)$, $\xi \in \Cr(X)$ and $A$ be an apartment containing the sector $Q(x, \xi)$. Then $Q(x,\xi)$
coincides with the intersection of all sectors $Q(y, \xi)$ where $y$ runs over the set of chambers of $A$
containing $x$.
\end{corollary}

\begin{proof}
By Lemma~\ref{lem:retraction:3}, there is a sequence $(R_n)$ of spherical residues of $A$ converging to $\xi$ in
$\Cr(A)$. By Proposition~\ref{Qxi} we have $ Q(x, \xi) = \bigcup_{k \geq 0} \bigcap_{n \geq k} \Conv(R, R_n)$.

Let $\mathcal C$ be the set of all chambers of $A$ containing $x$. Using the fact that the convex hull of two
residues is nothing but the intersection of all roots containing them, we deduce that for any $R_n$ we have
$$\Conv(x, R_n) = \bigcap_{y \in \mathcal C} \Conv(y, R_n).$$
It follows that
$$Q(x, \xi) = \bigcup_{k \geq 0} \bigcap_{y \in \mathcal C} \bigcap_{n \geq k} \Conv(y, R_n).$$
All it remains to show is thus that
$$\bigcup_{k \geq 0} \bigcap_{y \in \mathcal C} \bigcap_{n \geq k} \Conv(y, R_n) =
\bigcap_{y \in \mathcal C} \bigcup_{k \geq 0} \bigcap_{n \geq k} \Conv(y, R_n).$$ To establish this equality,
notice that the inclusion $\subset$ is immediate. The reverse inclusion follows similarly using the fact that
$\mathcal C$ is a finite set.
\end{proof}

We close this section with the following subsidiary fact.

\begin{lemma}\label{suitapp}
Let $x \in \Res(X)$,  $\xi \in \Cr(X)$ and $A$ be an apartment containing the sector $Q = Q(x, \xi)$. Then for
all $\alpha_1, \dots, \dots \alpha_n \in \Phi_A(\xi)$, the intersection
$$
Q(x,\xi)\cap (\bigcap_{i=1}^n \alpha_i)
$$
is non-empty.
\end{lemma}

\begin{proof}
In view of Corollary~\ref{cor:SectorFace}, it is enough to deal with the case $x$ is a chamber. Thus we assume
henceforth that $x \in \ch(X)$.

If the result is true with $n=1$, then it is true for any $n$ by a straightforward induction argument using
Remark~\ref{rem:sectors}.

We need to show that the sector $Q(x, \xi)$ penetrates any $\alpha \in \Phi_A(\xi)$.
We work by induction on the dimension of Davis' \cat realisation $|X|$ of $X$ (see \cite{Dav98}).  Recall that this dimension equals the maximal possible rank of a spherical residue of $X$. We call it the \textbf{dimension of $X$} for short.

Pick a sequence $(R_n)$ of spherical residues of $A$ converging to $\xi$ in $\Cr(A)$; such a sequence exists in view of Lemma~\ref{lem:retraction:3}. In order to simplify the notation, choose $(R_n)$ in such a way that $R_0 = x$. Since the desired result clearly holds if $\xi$ is an interior point, namely $\xi \in \Res(X)$, we shall assume that  $(R_n)$ goes to infinity.

Let $p_n \in |R_n|$. Upon extracting, the sequence $(p_n)$ converges to some point $\eta$ of the visual boundary $\bd |A| \subset \bd |X|$.  Let $\mathcal H_\eta$ denote the set of all walls $H$ of $A$ such that $\eta \in \partial |H|$. Equivalently some geodesic ray of $|A|$ pointing to $\eta$ is contained in a tubular neighbourhood of $|H|$. Let also $W$ denote the Weyl group of $X$ (which acts on $|A|$ by isometries),  and $W_\eta$ denote the subgroup generated by the reflections associated with the elements of $\mathcal{H}_\eta$. Recall from \cite{Deo} that  $W_\eta$ is a Coxeter group.

\smallskip
Let now $\alpha \in \Phi_A(\xi)$. If $x \subset \alpha$, then $Q(x, \xi) \subset \alpha$ by Proposition~\ref{prop:sector:roots} and we are done. We assume henceforth that $x$ is not contained in $\alpha$.

Recall that $x= R_0$. In particular $p_0 \in |x|$. Therefore, Proposition~\ref{prop:sector:roots} implies that  the geodesic ray $[p_0, \eta)$ is entirely contained in $|Q(x, \xi)|$. In particular, if this ray penetrates $|\alpha|$, then we are done. We assume henceforth that this is not the case. Since $p_n \in |\alpha|$ for any large $n$, this implies that the wall $\partial \alpha$ belongs to $\mathcal H_\eta$.

\medskip
We claim that there is some $R \in \Res(A)$ such that $R \subset \alpha$ and $R \subset \beta$ for all $\beta \in \Phi_A(x) \cap \Phi_A(\xi)$ such that $\partial \beta \in \mathcal H_\eta$. The proof of this claim requires to use the induction hypothesis for the thin building $A_\eta$ associated to the Coxeter group $W_\eta$. The walls and the roots of $A_\eta$ may be canonically and $W_\eta$-equivariantly identified with the elements of $\mathcal H_\eta$. This yields a well defined $W_\eta$-equivariant surjective map $\pi_\eta: \Res(A) \to \Res(A_\eta)$ which maps a residue $\sigma$ to the unique spherical residue which is contained in all roots $\phi$ containing $\sigma$ and such that $\partial \phi \in \mathcal H_\eta$. The way $\pi_\eta$ acts on $\ch(A)$ is quite clear: it identifies chambers which are not separated by any wall in $\mathcal H_\eta$.

We now verify that $\dim(A_\eta) < \dim(A)=\dim(X)$. Indeed, a spherical residue in a Davis complex is minimal (\emph{i.e.} does not contain properly any spherical residue) if and only if it coincides with  the intersection of all walls containing it. Now, given a spherical residue of maximal possible rank $\sigma$ in $A_\eta$, then on the one hand the intersection in $|A_\eta|$ of all the walls in $\mathcal H_\eta$ containing $|\sigma|$ coincides with $|\sigma|$, but on the other hand  the intersection of these same walls in $|A|$ is not compact since it contains a geodesic ray pointing to $\eta$. This shows that $\dim(A) > \dim(A_\eta)$ as desired.

We are now in a position to apply the induction hypothesis in $A_\eta$. Notice that, upon extracting,  the sequence $(\pi_\eta(R_n))$ converges in $\Cr(A_\eta)$ to a point which we denote by $\pi_\eta(\xi)$. Furthermore, we have
$$\Phi_{A_\eta}(\pi_\eta(\xi)) = \{ \beta \in \Phi_A(\xi) \; |\; \partial \beta \in \mathcal H_\eta\}.$$
By induction there is some $R' \in \Res(A_\eta)$ contained in both $\alpha$ and $Q(\pi_\eta(x), \pi_\eta(\xi))$. Let $R \in \Res(A)$ be any element such that $\pi_\eta(R) = R'$. In view of Proposition~\ref{prop:sector:roots}, we have $R \subset \alpha$ and $R \subset \beta$ for all $\beta \in \Phi_A(x) \cap \Phi_A(\xi)$ such that $\partial \beta \in \mathcal H_\eta$, which confirms the above claim.

\medskip
Pick now $p \in |R|$ any point supported by $R$ and consider the geodesic ray $\rho$ joining $p$ to $ \eta$. We shall prove that this ray penetrates $|Q(x, \xi)|$, from which the desired conclusion follows. Let $q_n = \rho(n)$ for all $n \geq 0$.

Suppose for  a contradiction that for all $n$, we have $q_n \not \in  |Q(x, \xi)|$. Then, in view of Proposition~\ref{prop:sector:roots} there exists a root $\alpha_n \in  \Phi_A(x) \cap \Phi_A(\xi)$ which does not contain $q_n$.

We claim that none of the $\partial \alpha_n$'s separate the ray $[x, \eta)$ from $[p, \eta)$. Indeed, if $\partial \alpha_n$ did, then it would belong to $\mathcal H_\eta$, which contradicts the definition of $R$.

Since $[x, \eta) \subset |Q(x, \xi)| \subset |\alpha_n|$ for any $n$, it follows in particular that for each $n$ there is some $n'$ such that $q_{n'} $ is contained in $\alpha_n$. Upon extracting, we may assume that either $n'>n$ or $n'<n$ for all $n$. In either case, it follows that the set  $\{\alpha_n\}$ is infinite and that for any $k$, the intersection $\bigcap_{n \leq k} -\alpha_k$ contains some point of $q'_k \in [p, \eta)$. In particular, when $k$ tends to infinity, the number of walls separating $q'_k$ from $[x, \eta)$ tends to infinity, which contradicts the fact that $[x, \eta)$ and $[p, \eta)$ are at finite Hausdorff distance from one another.
\end{proof}

\subsection{Incidence properties of sectors}

The goal of this section is to establish that two sectors pointing towards the same point at infinity have a non-empty intersection. This should be compared to the corresponding statement in the classical case of Euclidean buildings, see~\cite[2.9.1]{BT}.

\begin{proposition}\label{mmapp2}
Let $\xi$ be any point in $\Cr(X)$. Given any two residues $x, y \in \Res(X)$, there exists  $z\in \Res(X)$ such
that  $Q(z,\xi) \subset Q(x,\xi)\cap Q(y,\xi)$.
\end{proposition}

\begin{proof}
In view of Remark~\ref{rem:sectors}, it suffices to prove that the intersection  $Q(x,\xi)\cap Q(y,\xi)$
contains some spherical residue $z$. We proceed by induction on the root-distance $d(x,y)$.

If $x \subset y$ or $y \subset x$, the result is clear. Thus there is no loss of generality in assuming that the
open interval $]x, y[$ is non-empty, see Lemma~\ref{lem:interval}. Let $z \in \; ]x, y[$. By induction there
exists $a \in Q(x, \xi) \cap Q(z, \xi)$ and $b \in Q(y, \xi) \cap Q(z, \xi)$. Therefore, it suffices to show
that $Q(a, \xi) \cap Q(b, \xi)$ is non-empty. Since the sectors $Q(a, \xi)$ and $Q(b, \xi)$ are both contained
in $Q(z, \xi)$, it follows in particular that they are contained in a common apartment, say $A$. Let $\alpha_1,
\dots, \alpha_k$ be the finitely many elements of $\Phi_A(a) \cap \Phi_A(\xi) \setminus \Phi_A(b)$. By
Lemma~\ref{suitapp}, there is some spherical residue $R$ contained in $Q(b, \xi)$ as well as in each $\alpha_i$.

We claim that $R \in Q(a, \xi)$. If this were not the case, there would exist some $\alpha \in \Phi_A(a) \cap
\Phi_A(\xi)$ not containing $R$ in view of Proposition~\ref{prop:sector:roots}. The same proposition shows that
if $b \subset \alpha$, then $Q(b, \xi) \subset \alpha$ which is absurd since $R \subset Q(b, \xi)$. Therefore we
have  $b \not \subset \alpha$ or equivalently $\alpha \not \in \Phi_A(b)$. Thus $\alpha $ coincides with one of
the $\alpha_i$'s, and yields again a contradiction since $R  \subset \bigcap_{i=1}^k \alpha_i$. This confirms
the claim, thereby concluding the proof of the proposition.
\end{proof}

\subsection{Covering the combinatorial compactifications with apartments}\label{sec:sectors}

We are now able to prove Proposition~\ref{prop:phi}. In fact we shall establish the following more precise
version.

\begin{proposition}\label{prop:phi:bis}
Given $\xi \in \Cr(X)$,  we have the following.
\begin{itemize}
\item[(i)] There exists a sequence of spherical residues $(x_{0}, x_{1}, \dots)$ which penetrates and eventually remains in every sector
pointing to $\xi$, and such that $x_{n} = \proj_{x_{n}}(\xi)$ for all $n$.

\item[(ii)] Every such sequence converges to $\xi$.
\end{itemize}
\end{proposition}

\begin{proof}[Proof of Propositions~\ref{prop:phi} and~\ref{prop:phi:bis}]
In view of Proposition~\ref{prop:CcToCr} and the fact that a sequence as in (i) eventually remains in one
apartment, it suffices to prove Proposition~\ref{prop:phi:bis}.

\medskip \noindent
(i) Let $Q$ be some sector pointing to $\xi$ and $A$ be an apartment containing $Q$. Since $A$ is locally finite
and since any finite intersection of sectors pointing to $\xi$ is non-empty by Proposition~\ref{mmapp2}, it
follows that $Q$ contains a sequence $(x_n)$  of spherical residues which penetrates and eventually remains in
every sector \emph{contained in A} and pointing to $\xi$. Furthermore, upon replacing $x_{n}$ by $\proj_{x_{n}}(\xi)$ for all $n > 0$, we may and shall assume without loss of generality that $(x_{n})$ is eventually meets the \emph{interior} of every root $\alpha\in \Phi_{A}(\xi)$ such that $-\alpha \not \in \Phi_{A}(\xi)$. Notice that $\proj_{x_{n}}(\proj_{x_{n}}(\xi) )= \proj_{x_{n}}(\xi)$ in view of Corollary~\ref{cor:proj}.

If $Q'$ is any other sector pointing to $\xi$, then $Q
\cap Q'$ contains some sector by Proposition~\ref{mmapp2}. Therefore $(R_n)$ eventually penetrates and remains
in $Q'$ as desired.

\medskip \noindent
(ii) Let $(R_n)$ be a sequence of spherical residues which eventually penetrates and remains in every sector
pointing to $\xi$, and such that  $R_{n} = \proj_{R_{n}}(\xi)$ for all $n$.
Then the assumption on $(R_n)$ ensures that the sequence $(R_n)$ eventually remains on one side of every wall of $A$ (see Proposition~\ref{prop:sector:roots}). In
particular $(R_n)$ converges to some $\xi' \in \Cr(A)$ by Lemma~\ref{remarque}. By construction we have $\Phi_A(\xi) \subset
\Phi_A(\xi')$. Furthermore, since the sequence  $(R_{n})$  eventually leaves every root $\alpha$ of $A$ such that $-\alpha \not \in \Phi_{A}(\xi)$, we obtain in fact $\Phi_A(\xi) = \Phi_A(\xi')$.  Thus $\xi = \xi'$ by Lemma~\ref{lem:retraction:3}. Therefore we have $\proj_R(\xi) =
\proj_R(\xi') = \proj_R(R_n)$ for any sufficiently large $n$.

Since $R \in \Res(X)$ is arbitrary, we have just established that $R_n$ converges to $\xi$ in $\Cr(X)$ as
desired.
\end{proof}

We have seen in Lemma~\ref{remarque} that a sequence $(R_n)$ contained in some apartment $A$ converges to $\xi
\in \Cr(A)$ if and only if it eventually remains on one side of every wall of $A$. By
Proposition~\ref{prop:sector:roots}, the latter is equivalent to the fact that $(R_n)$ eventually penetrates and
remains in every sector of $A$ pointing to $\xi$. As we have just seen in the above proof, this implies that
$(R_n)$ converges in $\Cr(X)$. Thus we have proven the following:

\begin{corollary}\label{cor:convergence:criterion}
Let $(R_n)$ be a sequence of spherical residues contained in some apartment $A$. If $(R_n)$ converges in
$\Cr(A)$, then it also converges in $\Cr(X)$. In particular, it always admit a subsequence which converges in
$\Cr(X)$.\qed
\end{corollary}

\section{Horofunction compactifications}\label{sec:horofunction}

Let $Y$ be a \textbf{proper} metric space, \emph{i.e.} a metric space all of whose closed balls are compact. Given a base point $y_{0} \in Y$, we define $\mathscr F(Y, y_{0})$ as the space of $1$-Lipschitz maps $Y \to \R$ taking value $0$ at $y_{0}$. The topology of pointwise convergence (which coincides with the topology of uniform convergence since $Y$ is proper) turns $\mathscr F(Y, y_{0})$  into a compact space. To each $p \in Y$ we attach the functions
$$d_{p} : Y \to \R : y \mapsto d(p, y)$$
and
$$f_{p} : Y \to \R : y \mapsto d(p, y) - d(p, y_{0}).$$
Then $f_{p} $ belongs to $\mathscr F(Y, y_{0})$ and it is a matter of routine verifications to check that the map
$$Y \to \mathscr F(Y, y_{0}) : p \mapsto f_{p}$$
is continuous and injective. We shall implicitly identify $Y$ with its image. The closure of $Y$ in $\mathscr F(Y, y_{0})$ is called the \textbf{horofunction compactification} of $Y$. We denote it by $\Ch(Y)$.

Since $ \mathscr F(Y, y_{0})$ may be canonically identified with the quotient of the space of all $1$-Lipschitz functions by the $1$-dimensional subspace consisting of constant functions, it follows that the horofunction compactification is independent of the choice of the base point $y_{0}$.


It is well known that the horofunction compactification of a proper \cat space coincides with the visual compactification, see \cite[Theorem~II.8.13]{BH}. In the case of a locally finite building $X$, several proper metric spaces may be viewed as realisations of $X$: Davis' \cat realisation $|X|$ is one of them; the \textbf{chamber graph} (\emph{i.e.} the set of chambers endowed with the gallery distance) is another one; the set of spherical residues $\Res(X)$ endowed with the root-distance  is yet another. It should be expected that the respective horofunction compactifications of these metric spaces yield different spaces which may be viewed as compactifications of the building $X$. This is confirmed by the following.

\begin{theorem}\label{thm:horo}
The minimal combinatorial compactification of a locally finite building $X$ is $\Aut(X)$-equi\-variantly homeomorphic to the horofunction compactification of its chamber graph.

Similarly, the maximal combinatorial compactification of $X$ is $\Aut(X)$-equivariantly homeomorphic to the horofunction compactification of $\Res(X)$ endowed with the root-distance.
\end{theorem}

Abusing notation slightly, we shall denote by $\Ch(X)$ the horofunction compactification of  $(\Res(X), d)$, where $d$ denotes the root-distance. Since by definition the chamber graph embeds isometrically into $(\Res(X), d)$, it follows that $\Ch(\ch(X))$ is contained as a closed subset in $\Ch(X)$. This is confirmed by combining Proposition~\ref{prop:CcToCr} with Theorem~\ref{thm:horo}.

\begin{proof}[Proof of Theorem~\ref{thm:horo}]
In one sentence, the above theorem holds because the combinatorial bordifications are defined using combinatorial projections, and the latter notion may be defined purely in terms of the root-distance (see Corollary~\ref{cor:proj-interval}).  Here are more details.

We deal only with the maximal combinatorial compactification, the case of  the minimal one being similar but easier.

Let $(R_{n})$ and $(T_{n})$ be two sequences of spherical residues which converge to the same point $\xi \in \Cr(X)$. We claim that the sequences $(f_{R_{n}})$ and $(f_{T_{n}})$ both converge in $\Ch(X)$ and have the same limit.

Let $x \in \Res(X)$. We show by induction on the root-distance $d(x, y_{0})$ between $x$ and a base point $y_{0} \Res(X)$ that $f_{R_{n}}(x)$ and  $f_{T_{n}}(x)$ take the same value for all sufficiently large $n$. This implies the above claim.

Assume first that $x \subset y_{0}$. Let $A$ be an apartment containing $y_{0}$ and $R_{n}$. Consider the roots of $A$. Since $x \subset y_0$, we have
$$\begin{array}{rcl}
|\Phi_A(R_n, y_0)| - |\Phi_A(R_n, x) | &=& |\Phi_A(R_n, y_0) \setminus \Phi_A(R_n, x)| \\
&= &|\Phi_A(\proj_x(R_n), y_0)|,
\end{array}$$
where the last equality follows since every root containing $R_n \cup x$ also contains $(\proj_x(R_n)$, and conversely any root containing $\proj_x(R_n)$ but not $ y_0$ also contains $R_n \cup x$. By similar arguments, one obtains
$$|\Phi_A(x, R_n)| - |\Phi_A(y_0, R_n) | = |\Phi_A(x, \proj_x(R_n) \cup y_0)|,$$
where $\Phi_A(x, \proj_x(R_n) \cup y_0)$ denotes the set of all the roots of $A$ containing $x$ but neither $\proj_x(R_n)$ nor $y_0$. Remark that the projection $\proj_x(R_n)$ coincides with $\proj_x(\xi)$ for any sufficiently large $n$. This shows that
$$
\begin{array}{rcl}
f_{R_n}(x) & = & d_{R_n}(x) - d_{R_n}(y_0)\\
& = & \frac 1 2 ( |\Phi_A(R_n, x) | - |\Phi_A(R_n, y_0)| + |\Phi_A(x, R_n)| - |\Phi_A(y_0, R_n) |)
\end{array}$$
depends only on $x$, $y_0$ and $\proj_x(\xi)$ and for any large enough $n$. In particular this shows that the sequence $(f_{R_n})$ converges and its limit coincides with the limit of $(f_{T_n})$ as expected.

The same arguments apply to the case $x \supset y_0$.

Assume now that $x$ and $y_0$ are not contained in one another. Then the open interval $]x, y_0[$ is non-empty by Lemma~\ref{lem:interval}. Let $z$ be an element of this interval. By induction the sequences
$$n \mapsto  d_{R_n}(z) - d_{R_n}(y_0)$$
and
$$n \mapsto  d_{R_n}(x) - d_{R_n}(z)$$
both converge to some value which depends only on $\xi$. Since the sum of these sequences yields $(f_{R_n}(x))$, the desired result follows.

This provides a well defined $\Aut(X)$-equivariant map $\Cr(X) \to \Ch(X): \xi \mapsto f_\xi$. A straightforward modification of the above arguments also show that the latter map is continuous.

\medskip
Let now $(R_n)$ be a sequence of spherical residues such that
$(f_{R_n})$ converges to some $f \in \Ch(X)$. Given $x \in \Res(X)$,
the projection $\proj_x(R_n)$ coincides with the unique spherical
residue $\sigma$ containing $x$ and such that  $f_{R_n}(\sigma)$ is
minimal with respect to the latter property (see
Corollary~\ref{cor:proj-interval}). Since $X$ is locally finite, the
set $\St(x)$ is finite and we deduce from the above that
$\proj_x(R_n)$  takes a constant value, say $\xi_f(x)$, for all
sufficiently large $n$. Furthermore, if  $(T_n)$ were another
sequence  such that $(f_{T_n})$ converges to $f \in \Ch(X)$, then
the same arguments shows that $\proj_x(T_n)$ also converges to the
same $\xi_f(x)$. This shows that there is a well defined
$Aut(X)$-equivariant map $\Ch(X) \to \Cr(X) : f \mapsto \xi_f$ such
that $f_{\xi_f} = f$ and $\xi_{f_\xi} = \xi$ for all $f \in \Ch(X)$
and $\xi \in \Cr(X)$.

Thus $\Ch(X) $ and $ \Cr(X)$ are indeed $\Aut(X)$-equivariantly homeomorphic.
\end{proof}

\section{Group-theoretic compactifications}



\subsection{The Chabauty topology}

Let $G$ be a locally compact metrizable topological group and $\mathcal S(G)$ denote the set of closed subgroups of $G$. The reader may consult~\cite{BBK2} for an exposition of several equivalent definitions of the  \textbf{Chabauty topology} on $\mathcal S(G)$; this topology is compact (Theorem 1 of \S5.2 in \emph{loc.~cit}), metrizable and preserved by the conjugation action of $G$. The next proposition provides a concrete way to handle convergence in this space and could be viewed as yet another definition of the Chabauty topology.

\begin{lemma}\label{lem:ChabautyConvergence}
Let $F_n\in\mathcal S(G)$ for $n\geq 1$. The sequence $(F_n)$ converges to $F\in\mathcal S(G)$ if and only if the two following conditions are satisfied:

\begin{enumerate}[(i)]
\item For every sequence $(x_n)$ such that $x_n\in F_n$, if there exists a subsequence $(x_{\varphi(n)})$
converging to $x\in G$, then $x\in F$. \item For every element $x\in F$, there exists a subsequence $(x_n)$
converging to $x$ and such that $x_n\in F_n$ for every $n\geq 1$.
\end{enumerate}
\end{lemma}
\begin{proof}
See \cite[Lemma 2]{GuR}.
\end{proof}

\subsection{Locally finite groups}\label{locally finite}

Let $G$ be a topological group. The group $G$ is said \textbf{topologically locally finite} (or simply \textbf{locally finite} when there is no ambiguity) if  every finitely generated subgroup of $G$ is relatively compact. Zorn's lemma allows one to define the \textbf{locally finite radical} of $G$ (or \textbf{LF-radical}), denoted $\RLF(G)$,  as the unique subgroup of $G$, which is normal, topologically locally finite, and maximal for these properties. It may be shown that if $G$ is locally compact, then the closure of a locally finite subgroup is itself locally finite (see \cite[Lemma 2.1]{Cap}). In particular, in that case the LF-radical is a closed subgroup.

One also shows that if $G$ is locally compact, then $G$ is locally finite if and only if every compact subset of $G$ topologically generates a compact subgroup of $G$ (see~\cite[Lemma 2.3]{Cap}). In particular a locally compact topologically locally finite group is amenable.

\begin{example}
 Let $F$ be a non-archimedean local field, with absolute value $\vert\cdot\vert$ and ring of integers $\mathcal O_F$. In contrast with the archimedean case, the group $(F,+)$ is locally finite. Indeed, if $x_1,\dots,x_n$ are elements of $F$, then the subgroup they generate is included in the ball centered at the origin and of radius equal to the maximum of the absolute values of the $x_i$.

The group $F^\times$ is not locally finite: if $\vert x\vert\neq 1$ is different than one, then $x^n$ will leave every compact set as $n$ tends to $\pm\infty$. So $\RLF(F^\times)\subset \mathcal O_F^\times$, which is itself a compact group, and thus we have equality: $\RLF(F^\times)=\mathcal O_F^\times$.
\end{example}

\begin{example}\label{example:sl3-RLF}
 With the same notations as in the example above, let $P$ be the subgroup of $\SL_3(F)$ consisting of upper triangular matrices. The same argument as above proves that $\RLF(P)$ is included in the group $D$ of matrices of the form 
$$\left(\begin{matrix}
a\in\mathcal O_F^\times & \ast\in F &\ast\in F\\
0 & b\in \mathcal O_F^\times & \ast\in F\\
0 & 0 & (ab)^{-1}
 \end{matrix}\right).
$$
It turns out that $D$ itself is locally finite. Indeed, if $A_1,\dots,A_n$ are matrices in $D$, then a simple calculation shows that the absolute values of the elements of the upper diagonal elements in products and inverses of the $A_i$ are bounded. Then it follows that the upper right element is also of absolute value bounded. Hence $\RLF(P)=D$.

The group $D$ appears as an example of a limit group in \cite[\S6.2]{GuR}. Similar calculations also prove that the other limit groups which appear in \cite[\S6.2]{GuR}, such as the group of matrices of the form 
$$\left(\begin{matrix}
a&b & \ast\in F \\
c & d & \ast\in F\\
0 & 0 & (ad-bc)^{-1}
 \end{matrix}\right),
$$
with $\bigl(\begin{smallmatrix} a& b\\c & d\end{smallmatrix}\bigr)\in\GL_2(\mathcal O_F)$, are locally finite.
\end{example}

\subsection{Stabilisers of points at infinity}\label{stabilisateurs}
Let $X$ be a building and $G$ be a  locally compact group acting continuously by type-preserving
automorphisms on $X$ in such a way that the stabiliser of every spherical residue is compact. A special case in which the latter condition automatically holds is when the $G$-action on the \cat realisation $|X|$ is proper. In particular, this happens if $X$ is locally finite and $G$ is a closed subgroup of $\Aut(X)$.

\medskip
The goal of this section is to provide a description of the $G$-stabilisers of points in $\Cr(X)$.

\begin{lemma}\label{clfixe}
Let $x\in \Res(X)$ and $\xi\in \Cr(X)$. Then any element $g\in G$ fixing $x$ and $\xi$ fixes the sector $Q(x,\xi)$ pointwise.
\end{lemma}

\begin{proof}
It is clear that $g$ stabilises $Q(x,\xi)$.  Let $A$ be an apartment containing $Q(x, \xi)$ and $\rho$ be a retraction onto $A$ centred at some chamber $C$ containing $x$. Let $g_A: A \to A$ denote the restriction $\rho \circ g$ to $A$. Thus $g_A$ is a type-preserving automorphism of $A$ and all we need to show is that it fixes  $Q(x,\xi)$ pointwise. Let $y \in Q(x, \xi)$. If $y \subset x$, then $y$ is fixed by $g_A$ since $g_A$ is type-preserving. If $x \subset y$, then $y$ is contained in $\proj_x(\xi)$ by Corollary~\ref{cor:proj} and is thus fixed by $g_A$. Now, in view of Lemma~\ref{lem:interval}, the desired assertion follows from a straightforward induction on the root-distance $d(x, y)$.
\end{proof}

Recall that an element of a topological group is called \textbf{periodic} if the cyclic subgroup it generates is relatively compact.

\begin{lemma}\label{per=lf}
Let $\xi\in \Cr(X)$ and $G_\xi$ be its stabilizer in $G$. We have the following.
\begin{enumerate}[(i)]
\item The set of periodic elements $g\in G_\xi$ coincides with $\RLF(G_\xi)$.
\item For any apartment $A$ containing a sequence of spherical residues converging to $\xi$, we have
$$
\RLF(G_\xi)=\bigcup_{x\in \Res(X)}\Fix(Q(x,\xi))=\bigcup_{x\in \Res(A)}\Fix(Q(x,\xi)).
$$
\end{enumerate}
\end{lemma}

\begin{proof}(i)  Clearly, every element of $\RLF(G_\xi)$ is periodic. Conversely, let $g$ be a periodic element in $G_\xi$. Then $g$ fixes a point in $|X|$ by \cite[II.2.8]{BH}, and hence a spherical residue $x \in \Res(X)$. Now, given finitely many periodic elements $g_n$ and denoting by $x_n \in \Res(X)$ a $g_n$-fixed point, the group $\langle g_1, \dots, g_n \rangle$ fixes $\bigcap_{i=1}^n Q(x_i, \xi)$ pointwise by Lemma~\ref{clfixe}. In view of Proposition~\ref{mmapp2}, the latter intersection is non-empty. Thus $\langle g_1, \dots, g_n \rangle$ fixes a spherical residue and is thus contained in a compact subgroup of $G$. This shows in particular that the set of periodic elements forms a subgroup of $G$ which is locally finite. The desired conclusion follows.

\medskip \noindent
(ii) In view of Lemma~\ref{clfixe}, the equality  $\RLF(G_\xi) = \bigcup_{x\in \Res(X)}\Fix(Q(x,\xi))$ is a reformulation of (i). The inclusion $\bigcup_{x\in \Res(X)}\Fix(Q(x,\xi)) \supset \bigcup_{x\in \Res(A)}\Fix(Q(x,\xi))$ is immediate and the reverse inclusion follows from Proposition~\ref{mmapp2}.
\end{proof}

\begin{example}
 In the case of affine buildings, there are some points $\xi\in \Cr(X)$ such that the combinatorial sectors are usual sectors. In this case, the group $G_\xi$ and $\RLF(G_\xi)$ were already considered in \cite[\S4]{BT}, where they were denoted respectively $\mathfrak B$ and $\mathfrak B_0$.
\end{example}

Although we shall only need the following in the special case of sectors, it holds for arbitrary thin sub-complexes.

\begin{lemma}\label{fixateurQ}
Let $Y$ be a convex sub-complex of an apartment $A$ of $X$. Assume that $G$ acts strongly transitively on $X$. Then the pointwise stabiliser of $Y$ in $G$ is topologically generated by the pointwise stabilisers of those roots of $A$ which contain $Y$. Furthermore, this group acts transitively on the set of apartments containing $Y$.
\end{lemma}

\begin{proof}
As $Y$ is convex, it coincides with the intersection of roots in $A$ containing it. Let $H$ be the subgroup of $\Fix_G(Y)$ topologically generated by the pointwise stabilisers of such roots. We will first prove that $H$ is transitive on the set of apartments containing $Y$. Let $A'$ be such an apartment.

We shall repeatedly use the following fact which is easy to verify: \emph{since the $G$-action is strongly transitive, given two apartments $A_1, A_2$ which share a common half-apartment $\alpha$, there is an element $g \in G$ fixing $\alpha$ pointwise and mapping $A_1$ to $A_2$}.

This remark implies in particular that there is an element of $H$ which maps $A'$ to some apartment containing a chamber of $C$ of $A$ which meets $Y$. Therefore, it suffices to prove the desired assertion for the convex hull of $C \cup Y$. In other words, we may and shall assume that $Y$ contains some chamber $C_0$.

Let $C_1$ be a chamber of $A$ which meets $Y$ but is not contained in it. The above remark yields an element $g_1 \in H$ which maps $A' =: A'_0$ to some apartment $A'_1$ containing $C_1$. Proceeding inductively, one constructs sequences $(C_n)$,  $(A'_n)$ and $(g_n)$ such that:
\begin{itemize}
\item $C_n$ is a chamber of $A$ not contained in $Y_n := \Conv(Y \cup \{C_0, \dots, C_{n-1}\})$;

\item $A'_n$ is an apartment containing $Y_{n_1} \cup C_{n}$ and sharing a half-apartment with $A'_{n-1}$;

\item $g_n$ is an element of $H$ which maps $A'_{n-1}$ to $A'_n$.
\end{itemize}
Furthermore, these sequences are built in such a way that $A$ is covered by  $\bigcup_n Y_n.$ Thus for each  $C' \in \ch(A')$ there is some large $n$ such that $\rho_{A, C}(C') \subset Y_n$ and we deduce that $h_m(C')$ is contained in $A$ for all $m>n$, where  the sequence $(h_m)$ defined by $h_m = g_m \cdots g_1$. Since $H$ is compact, the sequence $(h_m)$ subconverges to some $h \in H$. Since the $G$-action is continuous, the above implies that $h$ maps $A'$ to $A$, as desired.

\medskip
It remains to show that $\Fix_G(Y)\subset H$. Let thus $g\in\Fix_G(Y)$ and set $A'=gA$. There exists some $h\in H$ such that $hA'=A$. Hence $hgA=A$ and since $hg$ fixes $Y$ pointwise, it is enough to show that the subgroup of $\Stab_G(A)$ which fixes $Y$ pointwise is contained in $H$. The latter subgroup is trivial if $Y$ contains a chamber. Otherwise it is generated by all the reflections of $\Stab_G(A)$ fixing $Y$. It is well known and easy to see how to express such a reflection as a product of three elements, which each fixes pointwise a root of $A$. Thus there reflections indeed belong to $H$, as desired.
\end{proof}

Combining Lemmas~\ref{per=lf} and~\ref{fixateurQ}, one obtains a description of the locally finite radical $\RLF(G_\xi)$ in terms of root groups.

\subsection{Description of the group-theoretic compactification}\label{chabauty}

We now assume that the building $X$ thick and locally compact, \emph{i.e.} of finite thickness. In particular the automorphism group $\Aut(X)$ of $X$, endowed with the topology of pointwise convergence, is locally compact and metrisable. Let $G< \Aut(X)$ be a closed subgroup consisting of type-preserving automorphisms.

We assume that $G$ acts \textbf{strongly transitively} on $X$, \emph{i.e.} $G$ acts transitively on the set of ordered pairs $(C,A)$ where $C$ is a chamber and $A$ an apartment containing $C$. (Throughout it is implicitly understood that the only system of apartments we consider the  full system.) In particular, the group $G$ is endowed with a Tits system, or $BN$-pair,
see \cite[Ch.~V]{Bro}. A basic exposition of Tits systems may be found in \cite[IV,\S2]{BBK}.

\medskip
The group-theoretic compactification of $X$ is based on the following simple fact.

\begin{lemma}
The map $\varphi : \Res(X) \to \mathcal S(G) : R \mapsto G_R$ which associates a residue $R$ to its stabiliser $G_R$ is continuous, injective, $G$-equivariant and has discrete image. In particular it is a homeomorphism onto its image.
\end{lemma}

\begin{proof}
Continuity is obvious since $\Res(X)$ is discrete. The fact that  $\varphi $ is equivariant is equally obvious. The injectivity follows since, by strong transitivity of the action, any two distinct residues have distinct stabilisers. It only remains to show that if some sequence $(R_n)$ of spherical residues is not asymptotically constant, then the sequence of stabilisers $G_{R_n}$ does not converge to some point of the image of $\varphi$. 

Let thus $(R_n)$ and $R$ be spherical residues  such that the sequence $(G_{R_n})$ converges to $G_R$. Suppose  for a contradiction that $R_n$ is not eventually constant.

Assume first that $R_n \not \supset R$ for infinitely many $n$. Then for each such $n$ there is an element $g_n \in G_{R_n}$ which fixes a vertex of $R$  but does not fix $R$. Clearly no subsequence of $(g_n)$ may converge to any element of $G_R$. On the other and since each $g_n$ fixes a vertex of $R$, it follows that the sequence $(g_n)$ is relatively compact and hence sub-converges in $g$.  In view of Lemma~\ref{lem:ChabautyConvergence}, this contradicts the fact that  $\lim_n G_{R_n} = G_R$.

Assume now that $R_n \supset R$ for all but finitely many $n$'s. Suppose for a contradiction that $R_n \supsetneq R$ for infinitely many $n$'s. Since $X$ is locally finite, this implies that there is a constant subsequence $R_{\psi(n)} = R'$ with $R' \supsetneq R$. Now the sequence $G_{R_{\psi(n)} }$ converges to both $G_R$ and $G_{R'}$, which implies the absurd equality $R=R'$. This finishes the proof.
\end{proof}

\begin{definition}
The closure of the image of $\varphi$ in $\mathcal S(G)$ is called the \textbf{group-theoretic compactification} of $X$. It is denoted by $\Cg(X)$.
\end{definition}

The main result of this section is the following.

\begin{theorem}\label{chab}
The group-theoretic compactification  $\Cg(X)$ is  $\Aut(X)$-equivariantly homeomorphic to the maximal
combinatorial compactification $\Cr(X)$. More precisely, a sequence $(R_n)$ of spherical residues converges to
some $\xi \in \Cr(X)$ if and only if the sequence of their stabilisers $(G_{R_n})$ converges to $\RLF(G_\xi)$ in
the Chabauty topology.
\end{theorem}

It follows in particular that the closure of the image of the chamber-set $\ch(X)$ under $\varphi$ is
$\Aut(X)$-equivariantly homeomorphic to the minimal combinatorial compactification $\Cc(X)$ (see
Proposition~\ref{prop:CcToCr}).

\begin{example}
The group-theoretic compactification of Bruhat--Tits buildings was already studied in \cite{GuR}. In particular they explicitely calculate the stabilizers and limit groups. In the case of the building associated to $\SL_3$ over a local field, there is some point $\xi$ such that $G_\xi=P$ is the group of upper triangular matrices. The limit group is thus the group calculated in Example~\ref{example:sl3-RLF}.
\end{example}

The proof of Theorem~\ref{chab} requires some additional preparations, collected in the following intermediate
results.

\begin{lemma}\label{stab}
Let $(R_n)$ be a sequence of spherical residues converging to a  point $\xi \in \Cr(X)$ and such that
the sequence $(G_{R_n})$ converges to some closed group $D$ in $\Cg(X)$. Then $D$ fixes~$\xi$.
\end{lemma}
\begin{proof}
Given $g\in D$ and $g_n\in G_{R_n}$ be a sequence which converges to $g$ (see
Lemma~\ref{lem:ChabautyConvergence}).  Let $\sigma \in  \Res(X)$. Then we have $g_n^{-1}.\sigma=g^{-1}.\sigma$
for $n$ large enough. Likewise, for $n$ large enough, $g.(\xi(g^{-1}\sigma)=g_n.(\xi(g^{-1}\sigma))$. Therefore
we have $(g.\xi)(\sigma)=g_n\xi(g_n^{-1}\sigma)=(g_n.\xi)(\sigma)$ for large $n$. Now, taking $n$ so large that
$\xi(\sigma) = \proj_\sigma(R_n)$ and $\xi(g^{-1}\sigma) = \proj_{g^{-1}\sigma}(R_n)$, we obtain sucessively
$$\begin{array}{rcl}
(g.\xi)(\sigma) &=& g.(\xi(g^{-1} \sigma))\\
& = & g_n.(\xi(g_n^{-1} \sigma))\\
& = & g_n.(\proj_{g_n^{-1} \sigma}(R_n))\\
& = & g_n. (\proj_{g_n^{-1} \sigma}(g_n^{-1}R_n))\\
& = & g_n. (g_n^{-1} \proj_{\sigma}(R_n))\\
& = & \proj_\sigma(R_n)\\
& = & \xi(\sigma).
\end{array}$$
Thus $g.\xi =\xi$ as desired.
\end{proof}

\begin{lemma}\label{chabcomb}
Let $(R_n)$ be a sequence of spherical residues converging to $\xi \in \Cr(X)$. Then the sequence  $(G_{R_n})$
converges in $\Cg(X)$ and its limit coincides with $\RLF(G_\xi)$.
\end{lemma}

\begin{proof}
Let $D$ be a cluster value of the sequence $(G_{R_n})$. It suffices to prove that $D=\RLF(G_\xi)$. This indeed
implies that $(G_{R_n})$ admits $D$ has its unique accumulation point, and hence converges to $D$.

Since $X$ is locally finite, the pointwise stabiliser of every bounded set of $X$ is open in $G$. Moreover,
since $G$ acts by simplicial isometries on $|X|$, it follows that every element acts either as an elliptic or as
a hyperbolic isometry. This implies that the set of elliptic isometries is closed in $G$. Notice that this set
coincides with the set of periodic elements of $G$.\footnote{It turns out that the latter fact is general and
does not depend on the existence of an action on a \cat space. Indeed,  by \cite[Theorem 2]{Wil} the set of
periodic elements is closed in any totally disconnected locally compact group.} Since every element of $D$ is
limit of some sequence of periodic elements by Lemma~\ref{lem:ChabautyConvergence}, it follows that $D$ itself
is contained in the set of periodic elements.  Lemmas~\ref{per=lf}  and~\ref{stab} thus yield $D\subset\RLF(G_{\xi})$.

In order to prove the reverse inclusion, pick $x\in X$ and let $A$ be an apartment containing $Q(x,\xi)$. By
strong transitivity, there exists some  $k_n \in G_x$ such that $k_nR_n\in A$. As $G_x < G$ is compact, we may
assume upon extracting that $(k_n)$ converges to some $k\in G_x$. Let $R'_n=k_n. R_n$. Then $(R'_n)$ is
contained in $A$ converges to $k.\xi$. Furthermore, $(G_{R'_n})$ converges to $kDk^{-1}$ in $\Cg(X)$.

The sequence $(R'_n)$ penetrates and eventually remains in every $\alpha \in\Phi_A(k.\xi)$. In particular, for
any sufficiently large $n$, the pointwise stabiliser $G_{(\alpha)}$ of $\alpha$ is contained in $G_{R'_n}$. By
Lemma~\ref{lem:ChabautyConvergence}, this implies that $G_{(\alpha)} < kDk^{-1}$. Conjugating by $k^{-1}$, we
deduce that for all $\alpha\in\Phi_A(\xi)$, we have $G_{(\alpha)} <  D$. In view of Lemma~\ref{fixateurQ}, this
shows that $G_{(Q(x,\xi))} < D$. The desired results follows since $G_{(Q(x,\xi))} = \RLF(G_\xi)$ by
Lemma~\ref{per=lf}.
\end{proof}

\begin{lemma}\label{combchab}
Let $(R_n)$ be a sequence of spherical residues. If the sequence $(G_{R_n})$ converges to $D\in \Cg(X)$, then
$(R_n)$ also converges in $\Cr(X)$.

Furthermore, for all $\xi, \xi'\in \Cr(X)$, we have $\xi = \xi'$ if and only if $\RLF(G_\xi)= \RLF(G_{\xi'})$.
\end{lemma}

\begin{proof}
Assume that $(G_{R_n})$ converges. If the sequence $(R_n)$ has two accumulation points $\xi, \xi' \in \Cr(X)$,
then Lemma~\ref{chabcomb} implies that $\RLF(G_\xi)= \RLF(G_{\xi'})$. Therefore, the Lemma will be proved if one
shows that the stabilisers of two distinct points of $\Cr(X)$ have distinct LF-radicals.

Given any $\xi \in \Cr(X)$ and $x \in \Res(X)$, the sector $Q(x, \xi)$ coincides with the fixed-point-set of
$G_{x, \xi}$ by Lemmas~\ref{clfixe} and~\ref{fixateurQ}. Furthermore Lemma~\ref{per=lf} implies that $G_{x, \xi}
= R_x$, where $R = \RLF(G_\xi)$. Thus $Q(x, \xi)$ is nothing but the fixed point set of $R_x$ for all $x \in
\Res(X)$. If follows that for any other $\xi' \in \Cr(X)$ such that $\RLF(G_{\xi'}) = \RLF(G_\xi)$, the
respective combinatorial sectors based at any $x \in \Res(X)$ and associated to $\xi$ and $\xi'$ coincide. In
view of Corollary~\ref{cor:proj}, this implies that $\xi = \xi'$.
\end{proof}

We are now ready for the following.

\begin{proof}[Proof of Theorem~\ref{chab}]
Consider now the map $$\Psi: \Cr(X) \to \mathcal S(G) : \xi \mapsto \RLF(G_\xi).$$ By
Proposition~\ref{chabcomb}, the map $\Psi$ takes its values in $\Cg(X)$. By Lemma~\ref{combchab}, it is
bijective. The $\Aut(X)$-equivariance is obvious. It only remains to show that $\Psi$ is continuous.

Let $(\xi_n)$ be a sequence of elements of $\Cr(X)$ converging to $\xi\in \Cr(X)$. We claim that every
accumulation point of $(\Psi(\xi_n))$ equals $\Psi(\xi)$. Let $D$ be such an accumulation point. Upon
extracting, we shall assume that $(\Psi(\xi_n))$ converges to $D$.

Since $\xi_n$ belongs to $\Cr(X)$, there exist some sequences $(x_m^n)_m$ of spherical residues such that
$(x_m^n)_m$ converges to $\xi_n$ for each $n$. A diagonal argument shows that the sequence $(x_m^m)_m$ converges
to $\xi$. By Lemma~\ref{chabcomb}, we deduce that $(\Psi(x_m^m))_m$ converges to $\Psi(\xi)$ while
$(\Psi(x_m^n))_m$ converge to $\Psi(\xi_n)$. Therefore, the sequence $(\Psi(x_m^m))_m$ converges to $\lim_n
\Psi(\xi_n)=D$. The desired equality $\Psi(\xi)=D$ follows.
\end{proof}

\section{Comparison to the refined visual boundary}\label{sec:stratification}

As opposed to the previous section, we do not assume here that $X$ be locally finite. In order to simplify the
notation, we shall often identify $X$ with its \cat realisation $|X|$. This will not cause any confusion. This section is devoted to the relationship between the combinatorial and visual compactifications and their variants.

\subsection{Constructing buildings in horospheres}\label{sec:buildingXxi}

Let $\xi \in \bd X$ be a point in the visual boundary of $X$. In this section we present the construction of a building $X_\xi$ which is canonically attached to $\xi$; it is acted on by the stabiliser $G_\xi$ and should be viewed as a structure which is `transverse' to the direction $\xi$. The construction goes as follows.

Let $\Axi$ denote the set of all apartments $A$ such that $\xi \in \bd A$. Let also $\HAxi$ denote the set of all half-apartments $\alpha$ such that the visual boundary of the wall $\partial \alpha$   contains $\xi$. In particular, every $\alpha \in \HAxi$ is a half-apartment of some apartment in $\Axi$.

Since any geodesic ray is contained in some apartment (see \cite[Theorem E]{CH}), it follows that the set $\Axi$ is non-empty. This is not the case for $\HAxi$, which is in fact empty when $\xi$ is  a `generic' point at infinity. We shall not try to make this precise.

\begin{lemma}\label{lem:2rays}
For all $A, A' \in \Axi$ and each $C \in \ch(A)$ and each geodesic
ray $\rho' \subset A'$ pointing to $\xi$, there exists an apartment
$A'' \in \Axi$ containing both $C$ and a subray of $\rho'$.
\end{lemma}

\begin{proof}
We work by induction on $d(C, \ch(A'))$. Let thus $C'$ be a chamber
of $A'$ at minimal possible distance from $C$ and let $C'=
C_0,C_1,\dots,C_n = C$ be a minimal gallery. The panel which
separates $C_0$ from $C_1$ defines a wall in $A'$, and there is some
half-apartment $\alpha$ of $A'$ containing a subray of $\rho'$. Then
$C_1\cup\alpha$ is contained in some apartment, and the desired
claim follows by induction on~$n$.
\end{proof}

Given $R \in \Res(X)$, let $R_\xi$ denote the intersection of all $\alpha \in \HAxi$ such that $R \subset \alpha$. Thus, in the case of chambers, the map $C \mapsto C_\xi$ identifies two adjacent chambers of $X$ unless they are separated by some wall $\partial \alpha$ with $\alpha \in \HAxi$.  We call two elements of  $\mathcal C_\xi $ \textbf{adjacent} if they are the images of adjacent chambers of $X$.

Let $W$ be the Weyl group of $W$. Choose an apartment $A \in \Axi$ and view $W$ as a reflection group acting on $A$. The reflections associated to half-apartments $\alpha$ of $A$ which belong to $\HAxi$ generate a subgroup of $W$ which we denote by $W_\xi$.
By the main result of \cite{Deo}, the group $W_\xi$ is a Coxeter group and the set $ \{C_\xi \; | \; C \in \ch(A) \}$ endowed with the above adjacency relation is $W_\xi$-equivariantly isomorphic to the chamber-graph of the Coxeter complex of $W_\xi$.

\begin{lemma}
The Coxeter group $W_\xi$  depends only on $\xi$ but not on the choice of the apartment~$A$.
\end{lemma}

\begin{proof}
By the above, it suffices to show that for any two $A, A' \in \Axi$, the adjacency graphs of $ \{C_\xi \; | \; C \in \ch(A) \}$ and $ \{C_\xi \; | \; C \in \ch(A') \}$ are isomorphic.

We claim that if the apartments $A$ and $A'$ contain a common chamber, then the retraction $\rho$ onto $A$ based at this chamber yields such an isomorphism. Indeed $\rho$ fixes $A \cap A'$ pointwise, and this intersection contains a ray pointing to $\xi$. This implies that for any half-apartment of $\alpha$ of $A'$, we have $\alpha \in \HAxi$ if and only $\rho(\alpha) \in \HAxi$. This proves the claim.

In view of Lemma~\ref{lem:2rays}, the general case of arbitrary  $A, A' \in \Axi$ follows from the special case that has just been dealt with.
\end{proof}

Keeping in mind the above preparation, the proof of the following result is a matter of routine verifications which are left to the reader. Lemma~\ref{lem:2rays} ensures that  two chambers of $X_\xi$ are contained in an apartment; this is the main axiom to check.

\begin{proposition}\label{immeuble}
The set $\mathcal C_\xi = \{C_\xi \; | \; C \in \ch(X) \}$ is the chamber-set of a building of type $W_\xi$ which we denote by $X_\xi$. Its full apartment system coincides with $\Axi$. The map $R \mapsto R_\xi$ is a $G_\xi$-equivariant map from $\Res(X)$ onto $\Res(X_\xi)$ which does not increase the root-distance.\qed
\end{proposition}

\begin{remark}\label{rem:dim(X)}
We have $\dim(X_\xi) < \dim(X)$. This was established implicitly in
the course of the proof of Lemma~\ref{suitapp}.
\end{remark}

\subsection{A stratification of the combinatorial compactifications}

By Proposition~\ref{immeuble} each point $\xi$ of the visual
boundary of $X$ yields a building $X_\xi$ and it is now desirable to
compare the respective combinatorial bordifications of $X$ and
$X_\xi$.

\begin{theorem}\label{thm:stratification}
For each $\xi \in \bd X$, there is a canonical continuous injective
$\Aut(X)_\xi$-equivariant map $r_\xi : \Cr(X_\xi) \to \Cr(X)$.
Furthermore, identifying  $\Cr(X_\xi) $ with its image, one has the
following stratification:
$$\Cr(X) = \Res(X) \cup \bigg(\bigcup_{\xi \in \bd X} \Cr(X_\xi) \bigg).$$
\end{theorem}

The following lemma establishes a first basic link between points at
infinity in the combinatorial bordification and points in the visual
boundary.

\begin{lemma}\label{lem:proj-visual}
Let $(R_n)$ be a sequence of spherical residues and let $(p_n)$
denote the sequence of their centres. Assume that $(R_n)$ converges
to some $f \in \Cr(X)$. Then $(p_n)$ admits convergent subsequences.
Furthermore, any accumulation point of $(p_n)$ lies in the visual
boundary of any combinatorial sector pointing to $f$.
\end{lemma}

It is not clear \emph{a priori} that $(p_n)$ subconverges in $X \cup
\bd X$ since $X$ need not be locally compact.

\begin{proof}
Fix a base point $p \in X$ and let $R \in \Res(X)$ denote its
support. For each $n$, the geodesic segment joining $p$ to $p_n$ is
contained in $\Conv(p, p_n)$ which is geodesically convex also in
the sense of \cat geometry. Therefore, in view of
Corollary~\ref{cor:sector:local}, it follows that for any $r>0$
there is some $N$ such that the geodesic segment $[p, p_n]$ lies
entirely in $Q(R, f)$ for all $n >N$. Since combinatorial sectors
are contained in apartments and since apartments are locally
compact, it follows that $(p_n)$ subconverges to some $\xi \in \bd
X$, and the above argument implies that the geodesic ray $[p, \xi)$
is entirely contained in the sector $Q(R, f)$.
\end{proof}

\begin{proof}[Proof of Theorem~\ref{thm:stratification}]
Let $\xi \in \bd X$, $f \in \Cr(X_\xi)$. We shall now define an
element $\widehat f : \Res(X) \to \Res(X)$ belonging to
$\prod_{\sigma \in \Res(X)} \St(\sigma)$. To this end, we proceed as
follows.

Consider the map $\Res(X) \to \Res(X_\xi) : \sigma \mapsto
\sigma_\xi$ which was constructed in Section~\ref{sec:buildingXxi}.
Let $\sigma \in \Res(X)$, let $\rho $ be a geodesic ray emanating
from the centre of $\sigma$ and pointing to $\xi$ and let $A$ be an
apartment containing $\rho$. Let $\Psi_A(\xi)$ denote the set of all
half-apartments $\alpha $ of $A$ such that $\alpha \not \in \Axi$
and $\alpha$ contains a subray of $\rho$. Notice that if $\alpha \in
\Psi_A(\xi)$, then $-\alpha \not \in \Psi_A(\xi)$.

Given $\tau \in \St(\sigma_\xi)$, there is a unique spherical
residue $\tau' \in \St(\sigma)$ such that $(\tau')_\xi = \tau$ and
that $\tau'$ is contained in every root $\alpha \in \Psi_A(\xi)$
containing $\sigma$. We denote this residue $\tau'$ by
$r_\xi(\tau)$. It is easy to see that the map $r_\xi :
\St(\sigma_\xi) \to \St(\sigma)$ does not depend on the choice of
the apartment~$A$.

Now we define  $\widehat f \in \prod_{\sigma \in \Res(X)}
\St(\sigma)$ by
$$\widehat f : \sigma \mapsto r_\xi( f(\sigma_\xi)).$$
Notice that the definition of $\widehat f$ does not depend on the choice of $A$.

We claim that $\widehat f$ belongs to $\Cr(X)$. Indeed, let $(x_n)$
of spherical residues of $X_\xi$ converging to $f$ and contained in
some apartment $A'$ of $X_\xi$ (see Proposition~\ref{prop:phi}). We
may view $A'$ as an apartment of $X$. Choose $R_n \in \Res(A')$ with
$(R_n)_\xi = x_n$ in such a way that the sequence $(R_n)$ eventually
penetrates and remains in the interior of every $\alpha \in
\Psi_{A'}(\xi)$. It is easy to see that such a sequence exists. If
follows from Lemma~\ref{remarque} and Proposition~\ref{immeuble}
that $(R_n)$  converges in $\Cr(A')$. The fact that $(R_n)$
converges in $\Cr(X)$ follows from
Corollary~\ref{cor:convergence:criterion}. The fact that $\lim_n
R_n$ coincides with $\widehat f$ follows from
Lemma~\ref{lem:retraction:3}. This proves the claim.

\medskip
We now show that $\Cr(X)$ admits a stratification as described
above. Let $h \in \Cr(X) \setminus \Res(X)$ and let $(R_n)$ be a
sequence of spherical residues contained in some apartment $A$ of
$X$ and converging to $h$ (see Proposition~\ref{prop:phi}). Upon
extracting, the sequence of centres of the $R_n$'s converges to some
$\xi \in \bd A$, and the sequence $((R_n)_\xi)_{n \geq 0}$ converges
in $\Cr(X_\xi)$. Let $h'$ denote its limit. Using the very
definition of the map $f \mapsto \widehat f$, one verifies that
$\widehat{h'} = h$, which yields the desired conclusion.
\end{proof}

\subsection{Comparison to the refined visual boundary}

Besides its own intrinsic \cat realisation, the building $X_\xi$
inherits a \cat realisation in a canonical way from $X$. This
follows actually from a general construction which may be performed
in an arbitrary \cat space and which attaches a transverse \cat
space to every point in the visual boundary. This construction was
described by Karpelevi\v{c} in the case of symmetric spaces; it was
introduced by Leeb~\cite{Lee} in the general context of \cat spaces
and used recently in~\cite{Cap} to study the structure of amenable
groups acting on \cat spaces. A brief description is included below.

Let $\xi\in\partial_\infty X$. We let $X^*_\xi$ denote the set of
geodesical rays $\rho$ pointing towards $\xi$. The set  $X^*_\xi$ is
endowed with a pseudo-distance defined by
$$d(\rho,\rho')=\inf_{t,t'\geq 0}d(\rho(t),\rho'(t')).$$

If $b_\xi$ is a Busemann function associated to $\xi$, and if the
parametrisation of $\rho$ and $\rho'$ is chosen so that
$b_\xi\circ\rho=b_\xi\circ\rho'$, then in fact
$d(\rho,\rho')=\lim_{t\to +\infty} (\rho(t),\rho'(t))$. This remark
justifies that $d$ is indeed a pseudo-distance.

Identifying points at distance~$0$ in $X^*_\xi$ yields a metric
space $X'_\xi$. There is no reason for this new space to be
complete;  its metric completion is denoted by $X_\xi$. There is a
canonical projection
$$\pi_\xi: X\to X_\xi$$
which associates to a point $x$ the (equivalence class of the)
geodesic ray from $x$ to $\xi$. It is immediate to check that
$\pi_\xi$ is $1$-Lipschitz.

Moreover, there is a canonical morphism $\varphi'_\xi: G_\xi\to
\Isom(X'_\xi)$, where $G = \Isom(X)$, defined by
$$\varphi'_\xi(g).\pi_\xi(x)=\pi_\xi(g.x).$$
The space $X_\xi$ is {\rm{CAT}}(0) (see \cite[Proposition 2.8]{Lee}). Furthermore the morphism $\varphi_\xi$ is continuous (see~\cite[Proposition 4.3]{Cap}).

\medskip	
As before, the space $X_\xi$ is \emph{transverse} to the direction
$\xi$. Since each transverse space $X_\xi$ admits its own visual
boundary, it is natural to repeat inductively the above construction
and consider sequences $(\xi_1, \xi_2, \dots)$ such that $\xi_{n+1}
\in \bd X_{\xi_1, \xi_2, \dots \xi_n}$. The next proposition shows
that this inductive process terminates after finitely many steps (in
the case of buildings, this should be compared to
Remark~\ref{rem:dim(X)}):

\begin{lemma}
There exists an integer $K\in\N$, depending only on $X$, such that
for every sequence $(\xi_1,\dots,\xi_K)$ with
$\xi_1\in\partial_\infty X$ and $\xi_{i+1}\in \bd
X_{\xi_1,\dots,\xi_i}$ we have $\bd
X_{\xi_1,\dots,\xi_K}=\varnothing$.
\end{lemma}
\begin{proof}
See the remark after \cite[Corollary 4.4]{Cap}.
\end{proof}

The following definition is taken over from~\cite{Cap}.

\begin{definition} The \textbf{refined visual boundary of level $k$} of $X$ is the set of all sequences
$(\xi_1,\dots,\xi_k,x)$, where  $\xi_1\in\partial_\infty X$ and
$\xi_{i+1}\in\partial_\infty X_{\xi_1,\dots,x_i}$  for all $1\leq
i\leq  k-1$ and $x\in X_{\xi_1,\dots,\xi_k}.$

The \textbf{refined visual boundary} of $X$ is the union over all
$k\in \N$ of the refined boundaries of level $k$. It is denoted by
$\bdfine X$.
\end{definition}

As mentioned earlier, in case the underlying space $X$ is a
building, the transverse space $X_\xi$ may be viewed as a \cat
realisation of the building $X_\xi$ constructed combinatorially in
Section~\ref{sec:buildingXxi}.  The following result shows that in
some sense, the refined visual boundary is a realisation of the
boundary at infinity of the combinatorial bordification.

\begin{theorem}\label{=raff}
Let $X$ be a building. Then there is an $\Aut(X)$-equivariant map
$\Theta: X \cup \bdfine X \to \Cr(X)$.
\end{theorem}
\begin{proof}
There is an $\Aut(X)$-equivariant surjective map $\Theta: X \to
\Res(X)$ which associates to each point its support. Recall that the
support of a point $x$ may be characterised as the unique spherical
residue contained in the intersection of all half-apartments
containing $x$.

Let now $\xi \in \bd X$; we consider the space $X_\xi$ both as a
\cat space as descrived above and as a building as described in
Section~\ref{sec:buildingXxi}. By induction on $\dim(X)$ (see
Remark~\ref{rem:dim(X)}) there is a well-defined
$\Aut(X)_\xi$-equivariant map $\bdfine X_\xi \to \Cr(X_\xi)$. Upon
post-composing with the map $r_\xi$ of
Theorem~\ref{thm:stratification}, we may assume that this map takes
it values in $\Cr(X)$. Since by definition, we have a partition
$$\bdfine X = \bigsqcup_{\xi \in \bd X} \bdfine X_\xi,$$
the existence of the desired map $\Theta$ follows.
\end{proof}

Notice that it is not clear \emph{a priori} (and not true in
general) that this map is surjective. Indeed, it might be the case
that the \cat space $X_\xi$ be reduced to a single point while the
associated building $X_\xi$ is a spherical building not reduced to a
single chamber. This happens for example of $X$ is a Fuchsian
building and $\xi$ is an end point of some wall.

\section{Amenability of stabilisers}

Let $X$ be a building. The following shows the relationship between
amenable subgroups of $\Aut(X)$ and the combinatorial bordification
$\Cr(X)$.

\begin{theorem}\label{moyennfix}
Let $G$ be a locally compact group acting continuously on $X$. Then
some finite index subgroup of $G$ fixes a point in $\Cr(X)$.

Assume in addition that the stabiliser in $G$ of every spherical
residue is compact. Then the stabiliser of any point of $\Cr(X)$ is
a closed amenable subgroup.
\end{theorem}

\begin{proof}
By~\cite[Theorem 1.4]{Cap} (see also
\cite[Theorem~1.7]{CapraceLytchak} in case $X$ is not locally
compact), the group $G$ has a finite index subgroup $G^*$ which
fixes a point in $X \cup \bdfine X$. Its image under the equivariant
map $\Theta$ of Theorem~\ref{=raff} is thus a $G^*$-fixed point in the
combinatorial bordification $\Cr(X)$.

Assume now that elements of $\Res(X)$ haves compact stabilisers in
$G$ and let $f \in \Cr(X)$. We shall prove by induction on $\dim(X)$
that the stabiliser $G_f$ fixes some point in the refined visual
bordification $X \cup \bdfine X$. The desired result on amenability
will then be provided by \cite[Theorem~1.5]{Cap} (see also the
remark following Theorem~1.1 in \emph{loc.~cit.} as well as
\cite[Theorem~1.7]{CapraceLytchak} for the non-locally compact
case).

If $f \in \Res(X)$, then $G_f$ fixes the centre of the residue $f$
and there is nothing to prove. Since the latter happens when $X$ is
has dimension~$0$, the induction can start and we assume henceforth
that $f$ is a point at infinity.

Notice that combinatorial sectors are closed and convex in the \cat
sense. Let $\mathcal Q_f$ denote the collection of all combinatorial
sectors pointing to $f$. By Proposition~\ref{mmapp2}, the set
$\mathcal Q_f$ forms a \textbf{filtering family} of closed convex
subsets, \emph{i.e.} any finite intersection of such sectors is
non-empty and contains such a sector. Since $f$ lies at infinity, it
follows that $\bigcap \mathcal Q_f$ is empty. It then follows from
\cite[Theorem~1.1]{CapraceLytchak} and
\cite[Proposition~1.4]{BalserLytchak} that the intersection of the
visual boundaries of all elements of $\mathcal Q_f$ admits a
canonical barycentre $\xi \in \bd X$ which is thus fixed by $G_f$.
In particular $G_f$ acts on the building $X_\xi$ transverse to
$\xi$.

We claim that $G_f$ fixes a point in $\Cr(X_\xi)$. In order to
establish it, notice first that by definition $\xi$ belongs to the
visual boundary of every apartment containing a sector in $\mathcal
Q_f$. Pick such an apartment $A$. Then $A$ may also be viewed as an
apartment of $X_\xi$ and its walls in $X_\xi$ is a subset of its
walls in $X$. By Lemmas~\ref{remarque} and~\ref{lem:roots} and
Corollary~\ref{cor:convergence:criterion}, it follows that $f$
determines a point $f' \in \Cr(X_\xi)$. Furthermore, since $A$
contains a subsector of every element of $\mathcal Q_f$, it follows
from Lemma~\ref{lem:retraction:3} that $f'$  is uniquely determined
by $f$. In particular $G_f$ fixes $f' \in \Cr(X_\xi)$ as claimed.

Since $\dim(X_\xi) < \dim(X)$ by Remark~\ref{rem:dim(X)}, it follows
from the induction hypothesis that $G_{f} < G_{f'}$ fixes a point in
the refined visual bordification $X_\xi \cup \bdfine X_\xi$. By definition, the latter embeds in the refined visual boundary
$\bdfine X$. Thus we have shown that $G_f$ fixes a point in the
refined visual boundary of $X$ as desired.
\end{proof}

\appendix
\section{Combinatorial compactifications of \cat cube complexes}

In this appendix, we outline how some of the above results may be
adapted in the case of finite-dimensional \cat cube complexes. Since
the arguments are generally similar but easier than in the case of
buildings, we do not include detailed proofs but content ourselves
by referring to the appropriate arguments in the core of the text.

Let thus $X$ be such a space. The $1$-skeleton $X^{(1)}$ induces a
combinatorial metric on the set of vertices $X^{(0)}$ which is
usually called the $\ell^1$-metric. In general it does not coincide
with the restriction of the \cat metric. The distance between two
vertices may be interpreted as the number of hyperplanes separating
them.

Let $\mathcal P$ denote the product of all pairs $\{h^+, h^-\}$ of
complementary half-spaces. Then there is a canonical embedding
$X^{(0)} \to \mathcal P$ which is defined by remembering on which
side of every wall a point lies. The closure of $X^{(0)}$ in
$\mathcal P$ is denoted by $\Cu(X)$. It is called the \textbf{Roller
compactification} or \textbf{ultrafilter compactification} of $X$;
see~\cite[\S3.3]{Guralnick} and references therein. It is a
natural analogue of the minimal combinatorial compactification of buildings introduced in the core of the paper. Notice that,
as opposed to the case of buildings, the space $\Cu(X)$ is compact
even if $X$ is not locally finite. The following result is due to
U.~Bader and D.~Guralnick (unpublished); it should be compared to
Theorem~\ref{thm:horo}.

\begin{proposition}
The ultrafilter compactification coincides with the horofunction
compactification of the vertex-set $X^{(0)}$ endowed with the
$\ell^1$ metric.
\end{proposition}

The following is an obvious adaption Lemma~\ref{remarque}; it is
established with the same proof.

\begin{lemma}\label{ap:remarque}
Let $(v_n)$ be a sequence of vertices.  Then the sequence $(v_n)$ in
$\Cu(X)$ if and only if for each wall $W$ there is some $N$ such
that the subsequence $(v_n)_{n >N}$ lies entirely on one side of
$W$.\qed
\end{lemma}

This allows one to associate with every $\xi \in \Cu(X)$  the set
$\Phi(\xi)$ of all half-spaces in which every sequence converging to
$\xi$ penetrates and eventually remains in. We define the
\textbf{combinatorial sector} based at a  vertex $v$ and pointing to
$\xi$ as the set
$$Q(v, \xi) = \bigcap_{h \in \Phi(v) \cap \Phi(\xi)} h.$$

The \textbf{(combinatorial) convex hull} of a set of vertices is
defined as the intersection of all half-spaces containing it. Having
this in mind, it is straightforward to prove that for all $v \in
X^{(0)}$ and any sequence $(v_n)$ of vertices converging to some
$\xi \in \Cu(X)$, we have
$$Q(v, \xi) = \bigcup_{k \geq 0} \bigcap_{n \geq k} \Conv(v, v_n),$$
compare Propositions~\ref{Qxi} and~\ref{prop:sector:roots}. The key
property of combinatorial sectors pointing to some $\xi \in \Cu(X)$
is that they form a filtering family:

\begin{proposition}\label{ap:filtering}
Let $v, v' \in X^{(0)}$ and $\xi \in \Cu(X)$. Then there exists some
vertex $v''$ such that $Q(v'', \xi) \subset Q(v, \xi) \cap Q(v',
\xi)$.
\end{proposition}

\begin{proof}
Use induction on $\dim(X)$ mimicking the proof of
Lemma~\ref{suitapp}.
\end{proof}

Assume for the moment that $X$ is locally finite; then the automorphism group $G = \Aut(X)$ is locally compact and we may as before consider the closure of the set of vertex-stabilisers in the Chabauty compact space $\mathcal S(G)$ of closed subgroups of $G$. Notice however that one should not expect the latter to coincide with the ultrafilter compactification in general: the most obvious reason for this is that the group-theoretic compactification need not be a genuine compactification if $G$ is to small --- for example if $G$ is discrete and torsion free, the group-theoretic compactification is a singleton. In fact, as opposed to the case of buildings, where the condition of strongly transitive actions is very natural, the transitivity properties one should impose on $G$ to make sure that the group-theoretic compactification is indeed a compactification of the vertex set do not seem natural at all. Therefore we shall not pursue this here and content ourselves with the following fact. 

\begin{proposition}
Let $(v_n)$ be a sequence of vertices of $X$ converging to some $\xi \in \Cu(X)$. Then the sequence of stabilisers $(G_{v_n})$ converges in the Chabauty topology and its limit coincides with $\RLF(G_\xi)$.
\end{proposition}

\begin{proof}
Let $D$ be an accumulation point of the sequence  $(G_{v_n})$. It suffices to show that $D = \RLF(G_\xi)$.

The proof of Lemma~\ref{stab} applies \emph{verbatim} to the present situation and ensures that $D \subset G_\xi$. Moreover, by similar arguments as in Lemma~\ref{per=lf}, one deduces from Proposition~\ref{ap:filtering} that the set of periodic elements of $G_\xi$ coincides with $\RLF(G_\xi)$. Since Lemma~\ref{lem:ChabautyConvergence} implies that $D$ consists of periodic elements, one obtains the inclusion $D \subset \RLF(G_\xi)$. 

In order to prove the reverse inclusion, consider an element $g \in \RLF(G_\xi)$. Then $g$ is periodic and hence it fixes some cube $C$ of $X$. Since the point $\xi$ determines exactly one side of each of the walls of $C$, it follows that $g$ fixes some vertex $v$ of $C$. In particular $g$ stabilises the sector $Q(v, \xi)$. It is easy to see by induction on the distance to $v$ that $g$ fixes pointwise all vertices contained in $Q(v, \xi)$. On the other hand, Lemma~\ref{ap:remarque} implies that the sequence $(v_n)$ penetrates and eventually remains in $Q(v, \xi)$. Therefore, we deduce that $g $ belongs to $G_{v_n}$ for any sufficiently large $n$. By Lemma~\ref{lem:ChabautyConvergence}, this implies that $g \in D$ as desired.
\end{proof}

We now drop off the assumption that $X$ be locally finite. The ultrafilter compactification may also be compared to the visual
boundary in a similar way as in Section~\ref{sec:stratification}; in
particular $\Cu(X)$ admits a stratification as in
Theorem~\ref{thm:stratification}. This may be used to established
the following by mimicking the proof of Theorem~\ref{moyennfix}.

\begin{theorem}
Every amenable locally compact group acting continuously on $X$ has
a finite index subgroup which fixes  some point in $\Cu(X)$.

Conversely, given a locally compact group $G$ acting continuously on
$X$ in such a way that every vertex has compact stabiliser, then the
stabiliser in $G$ of every point of $\Cu(X)$ is a closed amenable
subgroup.\qed
\end{theorem}

In the special case of a discrete group $G$, this last part was
established independently in~\cite{Niblo}. Remark that, as in the case of buildings, a closed subgroup $H<G$ is amenable if and only if $H / \RLF(H)$ is virtually Abelian (see~\cite{Cap}, as well as \cite[Theorem~1.7]{CapraceLytchak} for the non-locally compact case)

\bibliographystyle{alpha}
\bibliography{comb_cpt}
\end{document}